\documentclass{article}
\usepackage[T1]{fontenc}
\usepackage[utf8]{inputenc}
\usepackage{geometry}
\usepackage{amsmath}
\usepackage{amsthm}
\usepackage{thmtools,thm-restate}
\usepackage{algpseudocode}
\usepackage{caption}
\usepackage{float}
\usepackage{algorithm}
\usepackage[shortlabels]{enumitem}
\usepackage{stmaryrd}
\usepackage[dvipsnames]{xcolor} 
\usepackage[color=Orange!50!white, textwidth=22mm]{todonotes}
\usepackage[unicode=true, bookmarks=false, breaklinks=false,
 pdfborder={0 0 1},
 colorlinks=false,
 hidelinks=true]
 {hyperref}
\usepackage[noabbrev,capitalize,nameinlink]{cleveref}
\usepackage{comment}
\makeatletter
\theoremstyle{plain}
\newtheorem{thm}{\protect\theoremname}
\theoremstyle{plain}
\newtheorem{lem}[thm]{\protect\lemmaname}

\theoremstyle{remark}
\newtheorem{claim}[thm]{\protect\claimname}
\theoremstyle{plain}

\newtheorem{cor}[thm]{\protect\corollaryname}
\theoremstyle{plain}

\theoremstyle{plain}
\newtheorem{definition}[thm]{\protect\definitionname}
\theoremstyle{plain}
\newtheorem{conjecture}{Conjecture}

\usepackage{amssymb}

\makeatother

\providecommand{\claimname}{Claim}
\providecommand{\corollaryname}{Corollary}
\providecommand{\lemmaname}{Lemma}
\providecommand{\theoremname}{Theorem} 
\providecommand{\problemname}{Problem}
\providecommand{\definitionname}{Definition}
\providecommand{\observationname}{Observation}

\renewcommand{\epsilon}{\varepsilon}

\newcommand{\mB}{\mathcal{B}}
\newcommand{\mA}{\mathcal{A}}

\newcommand{\mQ}{\mathcal{Q}}
\newcommand{\mS}{\mathcal{S}}
\newcommand{\mC}{\mathcal{C}}

\newcommand{\mG}{\mathcal{G}}
\newcommand{\mF}{\mathcal{F}}

\algrenewcommand\algorithmicdo{:}
\algrenewcommand\algorithmicthen{:}
\newcounter{property}

\newcounter{stop}

\title{The typical structure of sets with large sumset}
\author{Rajko Nenadov \footnote{School of Mathematics and Statistics, University of Canterbury, New Zealand. Email: rajko.nenadov@canterbury.ac.nz Research supported by the New Zealand Marsden Fund.} \and Lander Verlinde \footnote{School of Mathematics and Statistics, University of Canterbury, New Zealand. Email: lander.verlinde@pg.canterbury.ac.nz}}
\date{}

\begin{document}

\maketitle
\begin{abstract}
Given $m \in \mathbb{N}$ and a $p$-random subset $A \subseteq \mathbb{N}$, we asymptotically determine $\log \Pr(|\mathbb{N} \setminus (A + A)| \ge m)$ for $p$ above the threshold for this property. As an application, we show that a uniformly at random chosen subset $A \subseteq \{1, \ldots, n\}$ of size $k$ such that $|A + A| \le \lambda k$, for $\lambda = o(k^{1/4} / \log^{1/2} n)$, with high probability belongs to an arithmetic progression of length $\lambda k / 2 + o(\lambda k)$. Both proofs are based on a bespoke container argument, exploiting the additive robustness of the considered sets.
\end{abstract}

\section{Introduction} 
Given a set $A \subseteq \mathbb N$, the set 
\[ A + A = \{ a+b : a,b \in A\} \]
is called the \emph{sumset} of $A$. The study of sumsets is one of the fundamental topics in additive combinatorics. For a given set $A$, the ratio $|A+A|/|A|$
is called the \emph{doubling factor} of $A$.  It is known that the doubling factor is minimised when $A$ is an arithmetic progression, and maximised when $A$ has no additive structure whatsoever. A foundational result by Freiman \cite{freiman1999structure} characterises the structure of a set with small doubling factor and relates it to a \textit{generalised arithmetic progression}. A generalised arithmetic progression of \emph{dimension~$d$} and \emph{volume~$N$} is a set of the form 
\[ A = \{ a + n_1 r_1 + \dots + n_d r_d: 0 \leq n_i < N_i \}, \]
where $N = \prod_{i=1}^d N_i$ and $a, r_1, \ldots, r_d \in \mathbb{N}$.
Freiman proved that if $A$ is a set such that $|A + A| \leq \lambda |A|$, then $A$ is contained in a generalised arithmetic progression of dimension at most $d(\lambda)$ and volume at most $f(\lambda) |A|$, where $d(\lambda)$ and $ f(\lambda)$ only depend on $\lambda$.  

Freiman's result was reproved by Ruzsa \cite{ruzsa1994generalized} with better bounds on $d(\lambda)$ and $f(\lambda)$, and ever since, significant further effort has been devoted to determining optimal bounds (see \cite{chang2002polynomial, sanders2012bogolyubov, sanders2013structure}). The current best bounds are by Schoen~\cite{schoen2011near}, 
who proved 
$$d(\lambda) = \lambda^{1+o(1)} \text{ and } f(\lambda) = \exp(\lambda ^{1+ o(1)}).$$ This is nearly optimal, which can be seen by considering a set without any additive structure,
requiring a generalised arithmetic progression with $d(\lambda) = \Theta(\lambda)$ and $f(\lambda) = 2^{\Theta(\lambda)}$.

The above results show that a set with small doubling factor has a lot of structure. Besides investigating the structure of the sets with small doubling factor, substantial research has been put towards counting such sets and further refining their \emph{typical} structure. Estimating the number of these sets has been inspired by the study of the clique number in random Cayley graphs \cite{green2005counting} and by the Cameron-Erd\H{o}s conjecture \cite{alon2014refinement,cameron1990number,green2004cameron,Sapzhenko2003}. 
Note that if $P \subseteq \{1, \ldots, n\}$ is an arithmetic progression of size $\lambda k / 2$ and $A \subseteq P$, then $|A+A| \leq \lambda k$. Therefore, there are at least $\binom{\lambda k /2}{k}$ sets $A \subseteq \{1, \ldots, n\}$ of size $|A| = k$ such that $|A + A| \le \lambda k$. In the proof of a refined version of the Cameron-Erd\H{o}s conjecture, Alon, Balogh, Morris and Samotij~\cite{alon2014refinement}  conjectured that these sets are essentially the main `source' of sets with doubling factor at most $\lambda$. That is, they conjectured that for every $\delta > 0$, there exists $c > 0$ such that if $\lambda \le c k/\log n$, for some $k \in \mathbb{N}$, then there are at most
$$
 \binom{\lambda k/2 (1 + \delta)}{k}
$$
sets $A \subseteq \{1, \ldots, n\}$ with $|A| = k$ and $|A + A| \le \lambda|A|$. We remark that the original conjecture stated $k \gg \log n$ and $\lambda < c k$, for a small constant $c > 0$. However, as pointed out by Morris (see \cite{campos2020number}), the best one can hope for is $\lambda = O(k / \log n)$. A first step towards this conjecture was taken by Green and Morris \cite{green2016counting}, who showed it to be true for $\lambda = O(1)$. Following improvements by Campos \cite{campos2020number} and Liu, Mattos and Szab\'{o} \cite{liu2024number}, the current range for which the conjecture is known to be true is $\lambda = o(k/\log^2 n)$. Further progress has recently been announced by Alon and Pham~\cite{alon2025random}.

A question that follows immediately is whether, in fact, a qualitative strengthening of the conjecture holds: Does $A$, chosen uniformly at random among all $k$-subsets of $\{1, \ldots, n\}$ subject to $|A + A| \le \lambda |A|$, typically belong to an arithmetic progression of length $\lambda k / 2 (1 + \delta)$? Note that this is significantly smaller than the `worst case' given by Freiman's Theorem. Analogous to the counting conjecture of Alon, Balogh, Morris and Samotij \cite{alon2014refinement}, we conjecture the following.
\begin{conjecture}
\label{conj: larger lambda}
    Let $n,k, \lambda \in \mathbb{N}$ be such that $\lambda = o(k/\log n)$. Let $A \subseteq \{1,\dots,n\}$ be a  set of size $k$ such that $|A + A| \leq \lambda k$, chosen uniformly at random among all such sets. Then, with high probability, $A$ is contained in an arithmetic progression of size $\lambda k /2 + o(\lambda k)$.
\end{conjecture}

Campos, Collares, Morris, Morrison and Souza \cite{campos2022probabilistic} showed this to be the case for $\lambda = o(k^{1/31})$. More precisely, they showed that with probability $1 - \varepsilon$, such a set $A$ is contained in an arithmetic progression of length 
\begin{equation} \label{eq:error_term}
    \frac{\lambda k}{2} + O\left(\lambda^2 \log\left(\varepsilon^{-1}\right) + \lambda^{32}\right).
\end{equation} 
We make further progress towards Conjecture \ref{conj: larger lambda}.

\begin{thm}
    \label{thm: final AP}
        Let $n,k, \lambda \in \mathbb{N}$ be such that $\lambda \geq 8$, $k \geq \log^2 n$, $\lambda = o(k^{1/4} /\log^{1/2} n)$. Let $A \subseteq \{1,\dots,n\}$ be a  set of size $k$ such that $|A + A| \leq \lambda k$, chosen uniformly at random among all such sets. Then, with high probability, $A$ is contained in an arithmetic progression of size $\lambda k /2 + o(\lambda k)$.
\end{thm}

A key ingredient in the proof of the result of Campos et al.~\cite{campos2022probabilistic}, and also a key ingredient in our proof of Theorem \ref{thm: final AP}, is an estimate of the probability that the sumset of a random set misses many elements. This problem, which is rather interesting for its own sake, was first studied by Green and Morris \cite{green2016counting}. We now discuss this in detail.

\paragraph{Sumsets of random sets.}

Given $0 \le p \le 1$, a \emph{$p$-random subset of $\mathbb{N}$} is a set $A$ obtained by including each number from $\mathbb{N}$ with probability $p$, independently of all other elements. Green and Morris \cite{green2016counting} showed that if $A \subset \mathbb{N}$ is a $p$-random subset with $p = 1/2$, then for any sufficiently large integer $m$ we have
\begin{equation} \label{eq:green_morris}
    \Pr\left( \left| \mathbb{N} \setminus (A + A) \right| \ge m \right) \le (1 - p)^{m/2 - o(m)}.
\end{equation}
It was also noted in \cite{green2016counting} that the same bound holds for any constant $p$. Campos, Collares, Morris, Morrison, and Souza  \cite{campos2022probabilistic} further showed that \eqref{eq:green_morris} holds for $p \gg m^{-1/8}$.
Note that \eqref{eq:green_morris} is optimal up to the $o(m)$ term in the exponent. Namely,
$$
    \Pr\left( A \cap \{1, \ldots, m/2\} = \emptyset \right) = (1 - p)^{m/2},
$$
and this event implies $|\mathbb{N} \setminus (A+A)| \ge m$. Our contribution is to fully settle the range of $p$.

\begin{thm}
\label{thm: missing naturals}
    There exists $C > 1$ such that the following holds. Let $m \in \mathbb{N}$ and $0 < \varepsilon < 1/2$, and suppose 
    $$
        \frac{C \log (1/\varepsilon)}{\sqrt{\varepsilon^3 m}} \le p \le 1/2. 
    $$
    Let $A \subseteq \mathbb{N}$ be a $p$-random subset. Then, 
    $$  
        \Pr\bigg( \big| \mathbb{N} \setminus (A+A) \big| \ge m  \bigg) < (1-p)^{m/2 - \varepsilon m}. 
    $$
\end{thm}

Note that $1/\sqrt{m}$ is a threshold for the event $|\mathbb{N} \setminus (A+A)| \ge m$. For $p \ll 1/\sqrt{m}$, with high probability we have $|A \cap \{1, \ldots, 2m\}| = o(\sqrt{m})$, thus 
$$
    |\{1,\ldots, 2m\} \setminus (A+A)| \ge 2m - |A \cap \{1, \ldots, 2m\}|^2 = 2m - o(m) > m.
$$
Therefore, the bound on $p$ in Theorem \ref{thm: missing naturals} is optimal, and so is the obtained probability up to the $\varepsilon m$ term in the exponent. It remains an interesting open problem to determine the optimal dependence on $\varepsilon$.

\paragraph{Structure of the paper.} Both Theorems  \ref{thm: final AP} and \ref{thm: missing naturals} have the same foundation: containers for \textit{sum-robust} sets. A pair of sets $(A,B)$ is sum-robust if the number of solutions to $a_1 + a_2 = b$ for $a_1, a_2 \in A$ and $b \in B$ does not decrease too much when removing a fraction of the elements in $A$ and $B$. We introduce this notion more precisely in Section \ref{sec: sum-robust sets}. One can think of this property as a supersaturation result, which will be the main driver of the container lemma proposed in the section. Indeed, given a sum-robust set of pairs, we show one can find a small $F \subseteq A$ such that either the intersection of $F+F$ and $B$ is large, or we can construct a large set $Q \subseteq A$ based solely on $F$. In either case, we are able to efficiently restrict where one of the sets in the sum-robust pair lives. 

Having this insight, we proceed by showing Theorem \ref{thm: missing naturals}. Instead of proving the theorem directly, we shall establish its finitary version, Theorem \ref{thm:interval_probability}, which estimates the probability of the event $|\{1, \ldots, 2n\} \setminus (A + A)| \ge m$, where $A$ is a $p$-random subset of $\{1, \ldots, n\}$. As we shall see in Section \ref{sec:final}, Theorem \ref{thm: missing naturals} then follows as an easy corollary. We estimate $\Pr(|\{1, \ldots, 2n\} \setminus (A + A)| \ge m)$ by considering the probability that $A + A$ misses many elements from \emph{end intervals} $\{1, \ldots, Km\}$ and $\{2n - Km + 1, \ldots, 2n\}$ (here `many' means $m - \varepsilon m$), and the probability that it misses few from the middle $\{Km + 1, \ldots, 2n - Km\}$ (where `few' means $\varepsilon m$). The first case is covered by Lemma \ref{lemma:prob_many_missing} in Section \ref{sec: missing many elements}, and follows from containers for sum-robust sets. The second case is covered by Lemma \ref{lemma:missing_few_elements} in Section \ref{sec: missing few elements}. This case is also resolved with a container argument, although it will require a different regularity notion. 

The final two sections are dedicated to showing Theorem \ref{thm: final AP}. To do so, we shall go through a couple of refinement steps. The first refinement is showing that with high probability, for a uniformly drawn $A \subseteq \{1, \ldots, n\}$ of size $|A| = k$ and $|A + A| \le \lambda |A|$, there exists a set $S$ and an arithmetic progression $P(S)$ such that $A \subseteq S$, $|P(S)| = \lambda k/2$ and $S \setminus P(S)$ is small. The proof of this refinement is built upon the containers for sum-robust sets, that are extended to also describe the structure of the contained sets. More precisely, for any set with bounded doubling, we find a small family of containers, that are in turn either very small, or are close to being an arithmetic progression. These \textit{Structural Containers} are derived in Section \ref{sec: structural containers}. We remark that the result of this kind was also obtained by Campos \cite{campos2020number} and Liu, Mattos and Szab\'o \cite{liu2024number}, however, we obtain improved bounds. Secondly, we show that, again with high probability, most of $A$ lies inside $P(S)$ and there are sufficiently many elements of $A$ living at the ends of $P(S)$. The final observation is that if there is an element outside of the extension of $P(S)$ of length $\lambda k/2 + o(\lambda k)$, the sum of that element and any element in the corresponding end part is not in the extended arithmetic progression. Hence $(A \cap P(S)) + (A \cap P(S))$ misses many elements, but this can only happen with very small probability (owing to Theorem \ref{thm:interval_probability}). The proof of Theorem \ref{thm: final AP} is made  precise in Section \ref{sec: typical structure}.

\section{Preliminaries}
Given $a, b \in \mathbb{R}$, we set 
$$
[a,b] = \{x \in \mathbb{N}: a \le x \le b \} \quad \text{ and } \quad (a, b) = \{x \in \mathbb{N}: a < x < b \}.
$$
The sets $(a, b]$ and $[a, b)$ are defined analogously. For $n \in \mathbb{R}$, we let $[n] := [1, n]$. Whenever we require that a certain value $x$ is an integer, we always take $\lceil x \rceil$. All logarithms are base 2.

Given $B \subseteq X \times Y$, for some sets $X$ and $Y$, for an element $x \in X$ we define
$$
    B(x, *) = \{ y \in Y : (x, y) \in B \}.
$$
For $y \in Y$, the set $B(*, y)$ is defined analogously. Given $X' \subseteq X$, we set $B(X', *) = \bigcup_{x \in X'} B(x, *)$. 

We define the \emph{$p$-random subset} of a countable set $X$ to be a set obtained by taking each element in $X$ with probability~$p$, independently of all other elements.

\paragraph{Additive combinatorics.}

The following definitions are used extensively throughout the paper.
\begin{definition}
    Given $X, Y \subseteq \mathbb{N}$, let
    $$
        \mS(X, Y) = \bigg\{ (x, x') \in X \times X : x + x' \in Y \; \text{ and } \; x \le x' \bigg\}.
    $$
    Given $a \in \mathbb{N}$, let 
    \[ \mS_a(X,Y) = \bigg\{ x \in X : a + x \in Y \bigg\}. \]
\end{definition}

Given a set $Y \subseteq \mathbb{N}$, we let $\overline{Y} = \mathbb{N} \setminus Y$.

\begin{thm}[Pollard's Theorem \cite{pollard1974generalisation}]
\label{thm: Pollard} 
    Let $\varepsilon < 1/2$ and let $X, Y \subseteq \mathbb{N}$ such that $|X| \geq ( \frac{1}{2} + \varepsilon) \left( |Y| +1 \right)$. Then 
    $$
        |\mS(X, \overline{Y})| \ge \varepsilon^2 |X|^2/2.
    $$
\end{thm}

Pollard's Theorem is generalised in multiple ways. We shall require two such generalisations, both due to Campos \cite{campos2020number}. The first one deals with the case where $X$ is slightly smaller than $(\frac{1}{2} + \varepsilon) |Y|$. 

\begin{lem}
\label{lemma: AP in Pollard}
Let $t \geq 0$ be an integer and $0 < \varepsilon < 2^{-10}$. If $X, Y \subseteq \mathbb N$ satisfy $|Y| \leq t$ and $(1-\varepsilon)t/2 \leq |X| \leq (1/2 + \varepsilon)t$, then one of the following holds:
\begin{enumerate}[({A}1)]
    \item \label{prop: no AP - many edges} $|\mS (X, \overline{Y})| \geq  2 \varepsilon^2 t^2 $;
    \item \label{prop: AP contained} there is an arithmetic progression $P$ with $|P| \leq t/2$ and a set $T \subseteq X$ with $|T| \leq 2^6 \varepsilon t$ such that $X \setminus T \subseteq P$.
\end{enumerate}
\end{lem}

The second statement is an asymmetric version of the theorem. While it was originally formulated for subsets of any abelian group, we only require a simplified version for the integers. 

\begin{lem}
\label{lem: asymmetric pollard}
Let $X,Y,Z \subseteq \mathbb{Z}$ be finite sets of integers and let $0 < \alpha < 1/2$. If $|X|+|Y| \geq (1+2\alpha) (|Z| +1)$, then there are at least $\alpha^2 |X| |Y|$ pairs $(x,y) \in X \times Y$ such that $x + y \notin Z$.

\end{lem}



\paragraph{Bounds and estimates}

We also apply three classical bounds on binomial coefficients.

\begin{lem}
  Let $a, b, c, d \in \mathbb{N}$. If $a > b \geq d$ and $a \geq c$, then
  \begin{equation} \label{eq:a-c_b-d}
      \binom{a-c}{b-d} \leq \left( \frac{a-c}{a} \right)^{b-d} \left( \frac{b}{a-b} \right)^d \binom{a}{b}.
  \end{equation}
  Whenever $a > b$ and $c \geq 0$, we have
  \begin{equation} \label{eq:a+c_b}
      \binom{a+c}{b} \leq \left( 1 + \frac{c}{a-b} \right)^b \binom{a}{b}.
  \end{equation}
  If $d \le c \le a / 4$, then
  \begin{equation} \label{eq:a_c-d}
      \binom{a}{c-d} \binom{b}{d} \leq \left( \frac{4bc}{ad} \right)^d \binom{a}{c}.
  \end{equation}
\end{lem}

\begin{proof}
    We start by showing \eqref{eq:a-c_b-d}, by showing an upper bound on the ratio $\binom{a-c}{b-d} / \binom{a}{b}$. Straightforward calculations show
    \[ \frac{\binom{a-c}{b-d}}{\binom{a}{b}} = \prod_{i=0}^{b-d-1} \frac{a-c-i}{b-d-i} \cdot \prod_{i=0}^{b-1} \frac{b -i}{a-i}
       = \prod_{i=0}^{b-d-1} \frac{a-c-i}{a-i} \cdot \prod_{i=b-d}^{b-1} \frac{1}{a-i} \cdot \prod_{i=0}^{d-1} b-i \leq \left( \frac{a-c}{a} \right)^{b-d} \left( \frac{b}{a-b} \right)^d. \]
    The proof of \eqref{eq:a+c_b} proceeds similarly:
    \[ \frac{\binom{a+c}{b}}{\binom{a}{b}} = \prod_{i=0}^{b-1} \frac{a+c-i}{b-i} \cdot \prod_{i=0}^{b-1} \frac{b-i}{a-i} = \prod_{i=0}^{b-1} \left(1 + \frac{c}{a-i}\right) \leq \left( 1 + \frac{c}{a-b}\right)^b. \]
  For \eqref{eq:a_c-d}, we use the standard binomial estimate $\binom{b}{d} < (e \cdot b/d)^d$, where $e$ denotes Euler's number. 
  \[ \binom{a}{c-d} \binom{b}{d} \leq \prod_{i = 0}^{c-d-1} \frac{a-i}{c-d-i} \cdot \left( \frac{b \cdot e}{d} \right)^d = \prod_{i=0}^{c-1} \frac{a-i}{c-i} \cdot \prod_{i=c-d}^{c-1} \frac{c-i}{a-i} \cdot \left( \frac{b \cdot e}{d} \right)^d \leq \left( \frac{4e \cdot bc}{3 \cdot ad} \right)^d \binom{a}{c},
  \]
  where the last inequality used $a - i \geq a -c \geq 3a/4$ for $i \in [c-d, c-1]$. The result follows, since $4e/3 < 4$.
\end{proof}

\section{Sum-robust sets}
\label{sec: sum-robust sets}
The key notion in the proofs of both Theorem \ref{thm: missing naturals} and Theorem \ref{thm: final AP} is that all sets we shall consider are robust in an additive sense. We call such sets \emph{sum-robust}, and this property shall be exploited to show a general container lemma. This lemma is the main driver of many subsequent container arguments.

\begin{definition}
    Let $X, Y \subseteq \mathbb{N}$. We say that the pair $(X,Y)$ is \emph{$\beta$-sum-robust} if for every $R_X \subseteq X$ and $R_Y \subseteq Y$ of size $|R_X|, |R_Y| \le \beta |X|$, we have
    $$
        |\mS(X \setminus R_X, Y \setminus R_Y)| \ge \beta^2 |X|^2.
    $$
\end{definition}

\begin{lem} \label{lem:robust_pair_container}
    Suppose $(X, Y)$ is $\beta$-sum-robust, for some finite $X, Y \subseteq \mathbb{N}$ and $\beta > 0$ such that $|X| > 16 \beta^{-3}$. Then for every $A \subseteq X$ of size $|A| \ge 2 \sqrt{|X| / \beta}$, there exist $F \subseteq A$  and $Q = Q(F) \subseteq X$ such that
    $$
        |F| \le 2 \sqrt{|X| / \beta},  \quad \left| (F + F) \cap Y \right| + |Q| \ge \beta |X| / 64 \quad \text{ and } \quad A \cap Q = \emptyset.
    $$
    Importantly, the set $Q(F)$ depends only on $F$ and not on the whole set $A$, and the same set $F$ can be obtained from any set $A'$ such that $F \subseteq A' \subseteq A$.    
\end{lem}
\begin{proof}  
    Consider an arbitrary set $A \subseteq X$ such that $|A| \ge 2 \sqrt{|X| / \beta}$. We construct the set $F$ using a two-phase procedure. Set $F := \emptyset$ and $R_X, R_Y := \emptyset$. Throughout the procedure we maintain $B \subseteq X \times Y$, initially also empty. Repeat the following for $\sqrt{|X| / \beta}$ steps:
    \begin{itemize}
        \item Let $a \in A \setminus F$ be the integer maximising $|\mS_a (X \setminus R_X,Y \setminus R_Y)|$, tie-breaking by picking the smallest.
        \item Add $a$ to $F$, and for each $x \in \mS_a(X \setminus R_X, Y \setminus R_Y)$ add $(x, a+x)$ to $B$.
        \item Set $R_X := \left\{ x \in X \colon |B(x, *)| \ge \sqrt{\beta |X|} \right\}$ and $R_Y := \left\{ y \in Y \colon |B(*, y)| \ge \sqrt{\beta |X|} \right\}$. 
    \end{itemize}
    Note that once some $x \in X$ becomes a part of $R_X$, it stays in the set until the end of the procedure and no pairs of the form $(x, y)$ are further added in $B$. The same holds for $y \in Y$. As $|B(x, *)|$ and $|B(*, y)|$ increase by at most one in each iteration, with room to spare we have
    \begin{equation} \label{eq:bound_Bxy}
        |B(x, *)|, |B(*, y)| \le 2 \sqrt{\beta |X|} \quad \text{ for every $x \in X$ and $y \in Y$}.
    \end{equation} 
    Another important thing to observe is that if we were to run the procedure on any $F \subseteq A' \subseteq A$ instead of $A$, we would have produced the same sets $F$, $R_X$, and $R_Y$. In particular, $R_X$, $R_Y$, and $B$ can be obtained from $F$ alone; in particular, they do not require the whole set $A$.
    
    We now distinguish two cases.
    
    \paragraph{Case 1: $|B| < \beta^{3/2} |X|^{3/2}/8$.} Then $|R_X|, |R_Y| < \beta |X|/8$, and for every $a \in A \setminus F$ we have 
    $$ 
        |\mS_a ( X \setminus R_X, Y \setminus R_Y)| <  \frac{\beta^2 |X|}{8}.
    $$
    Otherwise, since in each step we take $a \in A \setminus F$ which maximises $|\mS_a(X \setminus R_X, Y \setminus R_Y)|$ and the sets $R_X$ and $R_Y$ only potentially increase, we would have $|B| \geq |F| \cdot \beta^2 |X|/8 \geq \beta^{3/2} |X|^{3/2}/8$ -- a contradiction. 
    Set    
    $$ 
        Q := \{x \in X \setminus F \colon |\mS_x (X \setminus R_X, Y \setminus R_Y)| \geq \beta^2 |X| / 8\}.
    $$
    By the previous observation, we have $A \cap Q = \emptyset$. 
    
    Next, we claim that $|Q| + |R_X| + |F| > \beta |X|$, which implies $|Q| \ge \beta |X| / 2$. Suppose, towards a contradiction, that this is not the case. As $|R_Y| \le \beta |X| / 8$, by sum-robustness of $(X, Y)$ we conclude
    $$
        |\mS(X \setminus (Q \cup R_X \cup F), Y \setminus R_Y)| \ge \beta^2 |X|^2.
    $$
    However, by the definition of $Q$ we also have 
    $$
        |\mS(X \setminus (Q \cup R_X \cup F), Y \setminus R_Y)| < \beta^2 |X|^2 / 8,
    $$
    which is clearly a contradiction.

    As we have already observed that $R_X$ and $R_Y$ depend solely on $F$, the same holds for the set $Q$. Therefore, sets $F$ and $Q(F)$ satisfy all the properties.

    \paragraph{Case 2: $|B| \ge  \beta^{3/2} |X|^{3/2}/8$.} Set $F_0 := F$ and $F_1 := \emptyset$, and repeat the following for additional $\sqrt{|X| / \beta}$ steps: Take $a \in A \setminus (F_0 \cup F_1)$ to be the element which maximises $|B(a, *) \setminus B(F_0 \cup F_1, *)|$, tie-breaking by taking the smallest integer, and add it to $F_1$. The final set $F := F_0 \cup F_1$ is clearly of size $2 \sqrt{|X| / \beta}$, and it can be obtained by following the whole procedure up to this point with $A'$ in place of $A$, for any $F \subseteq A' \subseteq A$.
    
    Note that $B(F, *) \subseteq (F + F) \cap Y$. We distinguish two subcases.
    \begin{enumerate}[(a)]
    \item $|B(F, *)| \ge \beta |X| / 32$: In this case we have $|(F+F) \cap Y| \ge \beta |X|/32$, and we can simply set $Q = \emptyset$. All the properties are trivially satisfied.
    
    \item $|B(F, *)| < \beta |X| / 32$: Set $B' = B \setminus (X \times B(F, *))$. Since in each of the last $\sqrt{|X| / \beta}$ steps we have chosen $a \in A \setminus F$ which maximises $|B(a, *) \setminus B(F, *)|$ and the set $F$ increases in each iteration, for every $a \in A \setminus F$ we necessarily have 
    $$
        |B'(a, *)| < \frac{1}{32} \beta^{3/2} \sqrt{|X|}.
    $$
    Set
    $$
        Q := \left\{x \in X \colon |B'(x, *)| \ge \frac{1}{32} \beta^{3/2} \sqrt{|X|} \right\}.
    $$
    As $|B'(f, *)| = 0$ for $f \in F$, we conclude $A \cap Q = \emptyset$. By \eqref{eq:bound_Bxy}, we conclude
    $$
        |B'| \ge |B| - |B(F, *)| \cdot 2\sqrt{\beta |X|} > \frac{1}{16} \beta^{3/2} |X|^{3/2}.
    $$
    Applying \eqref{eq:bound_Bxy} once again, we conclude
    $$
        |B'| \le (|X| - |Q|) \cdot \frac{\beta^{3/2} \sqrt{|X|}}{32} + |Q| \cdot 2 \sqrt{ \beta|X|}.
    $$
    Therefore, $|Q| \ge \beta |X| / 64$.

    We have already established that $B$ can be obtained from $F_0$, and $F_0$ can be obtained from $F$. Therefore, the set $Q$ depends only on $F$ as well. 
    \end{enumerate}
\end{proof}

\section{Probability of missing many elements}
\label{sec: missing many elements}

The following lemma is the main result of this section. In the proof of Theorem \ref{thm:interval_probability}, we shall use it for end intervals, namely with $a = 1$ and $a = 2n - Td + 1$, where $T$ is a suitably large constant.

\begin{lem} \label{lemma:prob_many_missing}
    There exists $C \ge 1$ such that the following holds. Let $d \in \mathbb{N}$ and let $X = [a, a + Td -1]$, for some $T \ge 3/2$ and $a \in \mathbb{N}$. Let $\varepsilon > 0$, and let $A \subseteq X$ be a $p$-random subset where
    $$
        \frac{C T \log(T / \varepsilon)}{\sqrt{\varepsilon^3  d}} \le p \le 1/2.
    $$
    Then
    $$
        \Pr\bigg( \big| (X + X) \setminus (A + A) \big| \ge d \bigg) \le (1 - p)^{d/2 - 2 \varepsilon d}.
    $$
\end{lem}   

The key technical ingredient in the proof of Lemma \ref{lemma:prob_many_missing} is the following container-type lemma.

\begin{lem} \label{lemma:many_missing_container}
    There exists $L \ge 1$ such that the following holds. Let $d \in \mathbb{N}$ and let $X = [a, a + Td -1]$, for some $T \ge 1$ and $a \in \mathbb{N}$. Let $0 < \varepsilon < 1/2$ be such that $d \ge L T^2 / \varepsilon^3$. Then for any $A \subseteq X$ such that
    $$
        |(X + X) \setminus (A + A)| \ge d \quad \text{ and } \quad |A| \ge L T \sqrt{d / \varepsilon},
    $$
    there exist $F \subseteq A$ and $Q = Q(F) \subseteq X$ such that
    $$
         |F| = LT \sqrt{d / \varepsilon}, \quad |Q| \ge (1/2 - \varepsilon)d \quad \text{ and } \quad A \cap Q = \emptyset.
    $$    
    Importantly, the set $Q(F)$ depends only on $F$ and not on the whole set $A$.
\end{lem}

As mentioned, the proof of this lemma builds upon the containers for sum-robust sets.
The following lemma quantifies the sum-robustness of the pairs of sets we will encounter.

\begin{lem} \label{lemma:sum_robust_param}
    Let $d \in \mathbb{N}$, and let $X \subseteq \mathbb{N}$ be an interval of size $|X| = Td$, for some $T \ge 1$. Suppose that $Q \subseteq X$ is a set of size $|Q| = d/2 - \alpha d$ and that $Z \subseteq X + X$ is a set of size $|Z| = 2|X| - d - \zeta d$, for some $\alpha > 0$ and $\zeta \ge 0$, such that $(\alpha + \zeta)d \geq 2^6$. Then the pair $(X \setminus Q, (X+X) \setminus Z)$ is $\beta$-sum-robust for 
    $$
        \beta = \frac{\alpha + \zeta}{12T}.
    $$
\end{lem}
\begin{proof}

    Consider some $R_X \subseteq X \setminus Q$ and $R_Y \subseteq (X+X) \setminus Z$ with $|R_X|,|R_Y| \le |X \setminus Q| \cdot \beta  \le (\alpha + \zeta) d/12$. Set $A = X \setminus (Q \cup R_X)$ and $B = Z \cup R_Y$.  With Pollard's Theorem in mind, we first verify 
    \begin{equation} \label{eq:A_B}
        |A| \ge (1/2 + 2\beta) \left(|B|+1\right).
    \end{equation}
    Using the upper bound $|R_X|, |R_Y| \le (\alpha + \zeta) d/12$, inequality \eqref{eq:A_B} follows from
    $$
        |X| - (d/2 - \alpha d) - \frac{(\alpha + \zeta) d}{12} \ge (1/2 + 2\beta) \left(2|X| - d - \zeta d + \frac{(\alpha + \zeta) d}{12} +1 \right),
    $$
    which can be rewritten to 
    \[ \frac{7 \alpha d}{8} + \frac{3 \zeta d}{8} - \frac{1}{2} \geq 2 \beta \left(2 Td - d - \zeta d + \frac{(\alpha + \zeta) d}{12} +1 \right). \]
    Note that because $\alpha \leq 1/2$ and $\zeta \leq 2T -1$, we have 
    \[ \frac{(\alpha + \zeta) d}{12} \leq \frac{Td}{6} \quad \text{and} \quad 2 \beta \leq \frac{1}{3}. \]
    Thus the above inequality is satisfied whenever 
    \[ \frac{7 \alpha d}{8} + \frac{3 \zeta d}{8} - \frac{1}{2}  \geq \frac{13(\alpha + \zeta) d}{36} + \frac{1}{3}. \]
    This indeed holds, as $(\alpha + \zeta)d \geq 2^6$.

    By Pollard's Theorem and \eqref{eq:A_B} there are at least $2 \beta^2 |A|^2$ pairs $(a,a') \in A \times A$ such that $a \leq a'$ and $a + a' \not \in B$, that is, $a + a' \in (X + X) \setminus B$. As $\alpha \leq 1/2$ and $\zeta \le 2T$, we have $\beta < 1/4$, thus $|A| \geq 3|X \setminus Q|/4$. Therefore, 
    \[2 \beta ^2 |A|^2 \ge \beta^2 |X \setminus Q|^2.\]
    This verifies that the pair $(X \setminus Q, (X + X) \setminus Z)$ is indeed $\beta $-sum-robust.
\end{proof} 

We are now ready to prove Lemma \ref{lemma:many_missing_container}.

\begin{proof}[Proof of Lemma \ref{lemma:many_missing_container}]
    We prove the lemma with $L \geq 8 \sum_{k = 0}^\infty (1 - 2^{-11})^{k/2}$ a sufficiently large constant. Let $T \ge 1$ be such that $|X| = Td$. Set $F_0, Q_0, Z_0 = \emptyset$ and $i = 0$, and repeat the following until $|Q_{i}| \ge d/2 - \varepsilon d$:
    \begin{itemize}
        \item Let $\alpha_i > \varepsilon$ be such that $|Q_i| = d/2 - \alpha_i d$, and $\zeta_i \ge 0$ such that $|Z_i| = 2|X| - d - \zeta_i d$. Observe that $(\alpha_i + \zeta_i) d \geq L T^2/\varepsilon^3 \geq 2^6 $. Thus, by Lemma \ref{lemma:sum_robust_param}, the pair $(X \setminus Q_i, (X+X) \setminus Z_i)$ is $\beta_i $-sum-robust for $\beta_i = (\alpha_i + \zeta_i)/(12T)$.

        \item Set $X_i := X \setminus Q_i$. Let $F_i' \subseteq A$ be the set of size $|F_i'| \le 2 \sqrt{|X| / \beta_i}$ given by Lemma \ref{lem:robust_pair_container} applied with $A$, $X_i$ (as $X$) and $(X+X) \setminus Z_i$ (as $Y$), and let $Q_i' = Q(F_i')$ be the corresponding set. Set $F_{i+1} := F_i \cup F_i'$, $Z_{i+1} := Z_i \cup (F_{i+1} + F_{i+1})$, and $Q_{i+1} = Q_i \cup Q_i'$ (note that $Q_i$ and $Q_i'$ are disjoint) and increase $i$.
    \end{itemize}
    Once the procedure has terminated, which we shall briefly show indeed happens, set $F := F_i$ and $Q := Q_i$.

    \paragraph{The procedure is well-defined.} 
    We first verify that we can apply Lemma \ref{lem:robust_pair_container}. As already observed in the description of the procedure, the pair $(X_j, (X + X) \setminus Z_j)$ is $\beta_j $-sum-robust for $\beta_j = (\alpha_j + \zeta_j) / (12T)$, for every $0 \le j < i$. Next, we verify $|A| \ge 2 \sqrt{|X|  / \beta_j}$, which suffices since $|X_j| \leq |X|$. As $\beta_j \ge \alpha_j / (12T) \ge \varepsilon / (12T)$, which holds since otherwise we would have already finished the procedure, we have $2 \sqrt{|X|  / \beta_j} < 8 \sqrt{|X| T / \varepsilon}$. The desired inequality now follows by the assumption of the lemma, namely $|A| \ge L \sqrt{|X|T / \varepsilon}$. Finally, we need  $|X_j| \ge 16 / \beta_j^3$. Noting that $|X_j| \geq |X|/2 = Td/2$, and reusing $\beta_j \ge \varepsilon / (12T)$, this follows from the assumption $d \geq L T^2 / \varepsilon^3$ and taking $L$ large enough.

    Moreover, we have that $|(F'_i + F'_i) \cap ((X+X) \setminus Z_i)| \le |Z_{i+1}| - |Z_i|$, $|Q'_{i}| = |Q_{i+1}| - |Q_i|$. Thus, by Lemma~\ref{lem:robust_pair_container}, we have
    \begin{equation} \label{eq:Z_Q_increment}
        |Z_{j+1}| + |Q_{j+1}| \ge |Z_j| + |Q_j| + \beta_j |X| / 2^7.
    \end{equation}
    As $\beta_j > \varepsilon / (12T)$ and $|Z_{j+1}| + |Q_{j+1}| \le 2|X|$, the procedure eventually terminates.

    \paragraph{Size of $F_i$.}    
    From \eqref{eq:Z_Q_increment} we further conclude    
    $$
        \alpha_{j+1} + \zeta_{j+1} \le (\alpha_{j} + \zeta_j)(1 - 2^{-11}),
    $$
    which implies
    $$
        \alpha_j + \zeta_j \ge (1 - 2^{-11})^{-1}  (\alpha_{j+1} + \zeta_{j+1}) \ge (1 - 2^{-11})^{j-i+1} (\alpha_{i-1} + \zeta_{i-1}) \ge (1 - 2^{-11})^{j-i+1} \varepsilon.
    $$
    To avoid confusion, let us note that $\beta_j$ is only defined for $j \in [0,i-1]$.
    Therefore,
    $$
        |F_i| \le \sum_{j = 0}^{i-1} 2\sqrt{|X| / \beta_j} \le 2 \sqrt{Td} \sum_{k = 1}^{i} \sqrt{1 / \beta_{i-k}} \le 8T \sqrt{d / \varepsilon} \sum_{k = 0}^{i-1} (1 - 2^{-11})^{k/2} \leq L T \sqrt{ d / \varepsilon}.
    $$
    To finish, we keep adding the smallest element in $A \setminus F_i$ to $F_i$ until $|F_i| = L T \sqrt{ d / \varepsilon}$.
    \paragraph{$Q_i$ depends only on the set $F_i$.} By performing the same procedure with $A$ being $F_i$, due to the last property of Lemma \ref{lem:robust_pair_container} we end up producing the same $F_j$ and $Q_j$ in each iteration.
\end{proof}

To finish this subsection, we now derive Lemma \ref{lemma:prob_many_missing}.

\begin{proof}[Proof of Lemma \ref{lemma:prob_many_missing}]
    Let $L$ be the constant given by Lemma \ref{lemma:many_missing_container}. We can assume $d \ge L T^2 / \varepsilon^3$, as otherwise $p > 1$ (assuming $C \geq L$). Let 
    $$
        \mQ = \left\{ (F, Q(F)) : F \in \binom{X}{L T \sqrt{d / \varepsilon}} \right\}, 
    $$
    where $Q(F)$ is the set corresponding to $F$ in Lemma \ref{lemma:many_missing_container}. Note that $Q(F)$ is only defined for $F$ which corresponds to some $A \subseteq X$ such that $|(X+X) \setminus (A+A)| \ge d$ and $|A| \ge LT \sqrt{d / \varepsilon}$. For every other $F$ of size $LT \sqrt{d / \varepsilon}$, we simply set $Q(F) = X$. Recall that for every $(F, Q) \in \mQ$ we have $|F| = LT \sqrt{d / \varepsilon}$ and $|Q| \ge d/2 - \varepsilon d$.
    
    By Lemma \ref{lemma:many_missing_container}, if $|(X + X) \setminus (A + A)| \ge d$ then either $|A| < LT \sqrt{d / \varepsilon}$ or there exists $(F, Q) \in \mQ$ such that $F \subseteq A$ and $A \cap Q = \emptyset$. Therefore, we have the following bound:
    \begin{equation} \label{eq:A_misses_many}
        \Pr\bigg( \big| (X + X) \setminus (A + A) \big| \ge d \bigg) \le \Pr\bigg( |A| \le L T \sqrt{d / \varepsilon} \bigg) + \Pr\bigg( \bigcup_{(F, Q) \in \mQ} F \subseteq A \text{ and } A \cap Q = \emptyset \bigg).
    \end{equation}
    Choosing $C \ge 100L$, we have $\mathbb{E}[|A|] = |X|p \ge 100 L T \sqrt{ d / \varepsilon}$, with room to spare. Chernoff's bound gives
    \[ \Pr\bigg( |A| \le L T \sqrt{d / \varepsilon} \bigg) \leq \Pr\bigg( |A| \le \frac{\mathbb{E}[|A|]}{100} \bigg) \leq \left(100^{1/100} \cdot e^{-99/100}  \right)^{|X| p} \leq \left( e^{-94/100} \right)^{|X|p}. \]
    Now using that $|X| \geq 3d/2$ and $p \leq 1/2$, we obtain
    \[ \Pr\bigg( |A| \le L T \sqrt{d / \varepsilon} \bigg) \leq (1 - p)^d. \]
    The second term in \eqref{eq:A_misses_many} is estimated with a union bound:    
    \begin{align*}
        \Pr\bigg( \bigcup_{(F, Q) \in \mathcal{Q}} F \subseteq A \text{ and } Q \cap A = \emptyset \bigg) &\le \sum_{(F, Q) \in \mathcal{Q}} p^{|F|} (1 - p)^{|Q|} \le \sum_{(F, Q) \in \mathcal{Q}} p^{L T \sqrt{d / \varepsilon}} (1 - p)^{d/2 - \varepsilon d} \\
        &\le \left( \frac{e Td p}{L T \sqrt{d / \varepsilon}} \right)^{L T \sqrt{d / \varepsilon}} (1 - p)^{d/2 - \varepsilon d},
    \end{align*}
    where the penultimate inequality follows from a well-known estimate $\binom{x}{y} \le (ex / y)^y$. To finish the proof, it suffices to note that
    $$
        \left( \frac{e Td p}{L T \sqrt{d / \varepsilon}} \right)^{L T \sqrt{d / \varepsilon}} (1 - p)^{\varepsilon d/2} \le 1.
    $$
    This follows from the fact that the left-hand side is monotone decreasing in $p$ in the relevant range once $C$ has been chosen large enough, and so it achieves the maximum for $p = C T \log(T / \varepsilon) / \sqrt{\varepsilon^3 d}$. We omit routine calculation.
\end{proof}

\section{Probability of missing few elements}
\label{sec: missing few elements}
The following lemma is the main result of this subsection. It is important to note that the exponent in the upper bound increases linearly with $M$, and does not depend on $d$ (however, we do require that $d$ is large enough). This should not come as a surprise: the closer the number is to $n$, the more ways of writing it as $x + y$ for some $x, y \in [n]$, and hence the larger the probability that it is contained in $A + A$. 

\begin{lem} \label{lemma:missing_few_elements}
    There exists $K_0 \ge 1$ such that the following holds for any $K \ge K_0$. Let $n, d \in \mathbb{N}$ such that $d/ \log^3 d \geq 2^8 K$, and let $M = K d $. Let $A \subseteq [n]$ be a $p$-random subset for 
    $$
         \max \left\{ \frac{K \log K}{\sqrt{M}}, \frac{K^2 \log^2 d}{d} \right\} \le p \le 1/2.
    $$ 
    Then,
    $$
        \Pr\bigg( \big| [M, 2n - M] \setminus (A + A) \big| \ge d \bigg) \le (1 - p)^{M / 2^{11}}.
    $$
\end{lem}   

The idea of the proof is to consider events of the form $|Y_j \setminus (A + A)| \ge d / 2^j$, where $Y_j = [2^j M, 2^{j+1}M] \cup [2n - 2^{j+1} M, 2n - 2^j M]$. The main point of such a dyadic partition is that each $Y_j$ enjoys a certain regularity property. Namely, each $y \in Y_j$ can be represented in roughly the same number of ways as $x + x'$ for $x, x' \in [n]$. This property is exploited in Lemma \ref{lem:sum_regular_container}, which we use to prove the following lemma.

\begin{definition}
    Let $X, Y \subseteq \mathbb{N}$ be finite sets. We say that the pair $(X, Y)$ is \emph{$\kappa$-sum-regular}, for some $\kappa > 0$, if for each $y \in Y$ we have $|\mS(X, \{y\})| \ge \kappa |X|$.
\end{definition}

\begin{lem} \label{lem:missing_prob_sum_regular}
    For every $0 < \kappa \le 1$ there exists $C \ge 1$ such that the following holds. Let $d \in \mathbb{N}$ and let $X, Y \subseteq \mathbb{N}$ be finite sets such that $(X, Y)$ is $\kappa$-sum-regular, $d \ge 2 \sqrt{|X|} \ge C$ and $|Y| \geq 2d$. Let $A \subseteq X$ be a $p$-random subset, where 
    $$
        \frac{C \log(|Y|/d)}{\sqrt{|X|}} \le p \le 1/2.
    $$
    Then
    $$
        \Pr\bigg( \big|Y \setminus (A+A)\big| \ge d \bigg) \le (1 - p)^{\kappa |X| / 2^7}.
    $$
\end{lem}

We postpone the proof of Lemma \ref{lem:missing_prob_sum_regular} to the next section. Now, we use it to prove Lemma \ref{lemma:missing_few_elements}.

\begin{proof}[Proof of Lemma \ref{lemma:missing_few_elements}]
    Recall that $M = Kd$. Note that we can assume that $n > M$, as otherwise the statement is vacuous. 
    We first deal with the degenerate case $M < n \leq 2M$. Set $X = [n]$ and $Y = [M, 2n -M]$. Note that $(X,Y)$ is $(1/16)$-sum-regular, and that we may trivially assume that $d \leq |Y| = 2n - 2M +1$. If $|Y| \geq 2d$, one can easily see that $d \geq 2\sqrt{n}$, because $d \geq 2^8 K$. Thus we are able to apply Lemma \ref{lem:missing_prob_sum_regular} to obtain 
    \[ \Pr\bigg( \big|Y \setminus (A+A)\big| \ge d \bigg) \le (1 - p)^{n / 2^{11}} \leq (1 - p)^{M / 2^{11}} . \]
    If $d < |Y| < 2d$, we want to apply the same lemma, but we no longer satisfy the required bound on $|Y|$. Hence, we bound
    \[ \Pr\bigg( \big|Y \setminus (A+A)\big| \ge d \bigg) \leq \Pr\bigg( \big|Y \setminus (A+A)\big| \ge n - M \bigg). \]
    Then we have $|Y| \geq 2 (n-M)$, and $n-M \geq 2 \sqrt{n}$, since $n-M \geq (d-1)/2$ and $\sqrt{n} \leq \sqrt{d + M} = \sqrt{(1+K) d}$. Thus the bound follows from $d \geq 2^8 K$. So,
    \[ \Pr\bigg( \big|Y \setminus (A+A)\big| \ge n - M \bigg) \leq (1 - p)^{M / 2^{11}} .  \]
    By \cite[Lemma 4.3]{campos2022probabilistic}, we have
    \[
    \Pr\bigg( (2M/p, 2n - 2M/p) \not \subseteq A + A \bigg) \le \frac{8}{p^2} (1 - p^2)^{M/(2p)} \leq  \frac{8}{p^2} (1 - p)^{M/4},
    \]     
    where the last inequality follows from $1-p^2 \leq (1-p)^{p/2}$ for $0 \leq p \leq 1$. Note that the previous statement vacuously holds for $2M/p \ge n$. 
        
    For an integer $j \ge 0$, set $Y_j = [2^j M, 2^{j+1}M) \cup (2n - 2^{j+1}M, 2n - 2^j M]$ and $X_j = [1, 2^{j+1}M ) \cup (n - 2^{j+1}M, n]$. Observe that $(X_j, Y_j)$ is $(1/8)$-sum-regular for every $j \ge 0$.
    We define two possible pairs of final intervals: 
    set 
    \[ \widehat{Y} := [2^k M , 2M/p] \cup [2n - 2M /p, 2n - 2^k M] \quad \text{and} \quad \widehat{X} = [2M/p] \cup [n - 2M/p, n]. \]
    and 
    \[ \widetilde{Y} := [2^k M, 2n - 2^k M] \quad \text{and} \quad \widetilde{X} = [n]. \]
    We now distinguish two cases:
    \begin{itemize}
        \item If $2M / p < n$, let $k = \lfloor \log_2 (2p^{-1}) \rfloor -1 $, and set $(X_k, Y_k) = (\widehat{X}, \widehat{Y})$.
        \item Otherwise, let $k < \lfloor \log_2(2p^{-1})\rfloor$ be the largest integer so that $[2^{k+1} M, 2n - 2^{k+1} M] \neq \emptyset$. Set $(X_k, Y_k) = (\widetilde{X}, \widetilde{Y})$. Since $2^{k+1} M \leq n < 2^{k+2} M$, we have $2^{k+1}M \leq |Y_k| \leq 2^{k+3} M $.        
    \end{itemize}
    This ensures that $Y_0, \ldots, Y_k$ form a partition of 
    \[ I := \left([M, 2M/p] \cup [2n - 2M/p, 2n - M] \right) \cap [M, 2n - M]. \]
    Note that $(X_k, Y_k)$ is $(1/16)$-sum-regular.    
    Set $d_j = \max \{ d/2^{j+2}, d/(2\log d)\}$. Observe that if 
    $$
        \bigg| I \setminus (A + A) \bigg| \ge d,
    $$
    then by the pigeonhole principle (recall $k < \log d$) there exists $j \in \{0, \ldots, k\}$ such that 
    \begin{equation} \label{eq:Y_j_misses}
        |Y_j \setminus (A + A)| \ge d_j .
    \end{equation}
    Note that $Y_j \setminus (A + A) = Y_j \setminus (A_j + A_j)$, where $A_j = A \cap X_j$. By Lemma \ref{lem:missing_prob_sum_regular} applied with the pair $(X_j, Y_j)$, we have  
    \[\Pr\bigg( \big|Y_j \setminus (A_j+A_j) \big| \ge d_j \bigg) \le (1 - p)^{|X_j| /2^{11}} \]
    Let us briefly show that the use of Lemma \ref{lem:missing_prob_sum_regular} is indeed justified. Consider some $j \in [0,k]$:
    \begin{itemize}      
        \item Recall that $(X_j, Y_j)$ is $(1/16)$-sum-regular, often with room to spare.
        \item Since $|X_j| \geq 2^{j+1} M \geq 2^{j+1} K_0$, we pick $K_0$ large enough so that $|X_j| \geq C$, where $C$ is the constant given by the lemma for $\kappa = 1/16$. 
        \item $|Y_j| \geq  2^{j+1} M \ge 2d \ge 2d_j$.
        \item We now verify $d_j \geq 2 \sqrt{|X_j|}$. Note that $|X_j| \leq  2^{j+2} M$ for all $j < k$, and the same holds for $\widetilde{X}$. Furthermore, $|\widehat{X}| \le 4M/p +2$. If $j < \log \log d$ then $d_j = d/2^{j+2} \ge d /(2 \log d)$, thus the inequality holds whenever $d/\log^{3/2} d \geq 16 \sqrt{M}$. This is equivalent to $d/\log^3 d \geq 2^8 K$, which holds by the assumption of the lemma. If $\log \log d \leq j \le k \le \log(2p^{-1})$, then the inequality holds if $p \geq 2^6 M \log^2 d/ d^2$. Again using that $M = Kd$, this follows from $p \geq K^2 \log^2 d/d$. 
        \item Lastly, we need $p \geq C \log(|Y_j|/d_j)/\sqrt{|X_j|}$, where $C$ is the constant from Lemma \ref{lem:missing_prob_sum_regular} corresponding to $\kappa = 1/16$. Let $j < k$. Since $|Y_j| \leq  2^{j+3} M$, $|X_j| \geq 2^{j+1} M$, and $d_j \geq d/2^{j+2}$, it suffices to show that
        \[ p \geq \frac{C \log (2^{2j+5} M/d)}{ \sqrt{2^{j+1} M}}. \] 
        Assuming $K = M/d \ge K_0$ is large enough, this is a decreasing function in $j \ge 0$. Thus, this is satisfied as  $p \geq K \log K / \sqrt{M}$ and $K_0$ is sufficiently large with respect to $C$. The same bound is true for $\widetilde{X}, \widetilde{Y}$, and similar analysis holds for $|\widehat{X}| = 4M/p+2$ and $|\widehat{Y}| = 2(2M/p - 2^k M) +2$.
    \end{itemize}
    
    Therefore, 
    $$
        \Pr\bigg( \big|I \setminus (A+A) \big| \ge d \bigg) \le \sum_{j = 0}^k \Pr\bigg( \big|Y_j \setminus (A_j+A_j) \big| \ge d_j \bigg) \le \sum_{j = 0}^k (1 - p)^{|X_j|/2^{11}} \le (1 - p)^{M/2^{10}}.
    $$
    Finally, we conclude
    $$
          \Pr\bigg( \big| [M, 2n - M] \setminus (A + A) \big| \ge d \bigg) \le \Pr\bigg( ([M, 2n - M] \setminus I) \not \subseteq A + A \bigg) + \Pr\bigg( \big|I \setminus (A+A) \big| \ge d \bigg) \le (1 - p)^{M/2^{11}}.
    $$

\end{proof}

\subsection{Sum-regular sets}

The following lemma is the key ingredient in the proof of Lemma \ref{lem:missing_prob_sum_regular}. It is an analogue of Lemma \ref{lem:robust_pair_container}, and while the proofs are similar there are certain important differences.

\begin{lem} \label{lem:sum_regular_container}
    For every $0 < \kappa  \leq 1$, there exists $L > 0$ such that the following holds. Let $(X, Y)$ be a $\kappa$-sum-regular pair, for some $X, Y \subseteq \mathbb{N}$ with  $|X| \ge L$. Let $2 \sqrt{|X|} \leq d \leq |Y|/2$ be an integer, and suppose $A \subseteq X$ is such that
    $$
        |A| \geq L \log (|Y|/d) \sqrt{|X|} \quad \text{ and } \quad |Y \setminus (A + A)| \ge d.
    $$
    Then there exist $F \subseteq A$ and $Q = Q(F) \subseteq X$, such that 
    $$
        |F| = L \log(|Y| / d) \sqrt{|X|}, \quad |Q| \ge \kappa |X| / 64 \quad \text{ and } \quad A \cap Q = \emptyset.
    $$
    Importantly, $Q$ depends solely on $F$ and not on the whole set $A$.
\end{lem}
\begin{proof}    
    Consider some $A$ which satisfies the assumption of the lemma. We construct the set $F$ and, along the way, the accompanying set $Q$, using a three-phase procedure. 

    \paragraph{Phase I.} Set $Y_0 = Y$ and $F_0 = \emptyset$. As long as there exists a set $F' \subseteq A$ of size $|F'| \le 2\sqrt{|X|}$ such that $|(F' + F') \cap  Y_0| \ge \kappa |Y_0| / 16$, take the lexicographically smallest such $F'$ and set $F_0 = F_0 \cup F'$ and $Y_0 := Y_0 \setminus (F_0 + F_0)$. 
    
    Since $|Y_0| \ge |Y \setminus (A + A)| \ge d$ and $|Y_0|$ decreases by a multiplicative factor of $1 - \kappa/16$ in each step, at the end of Phase~I we have $|F_0| \le L' \log(|Y|/d) \sqrt{|X|}$, where $L'$ depends only on $\kappa$. As $|X| \ge |A| \ge L \log(|Y|/d) \sqrt{|X|}$ where, say, $L \ge 4L' / \kappa$, we conclude $|F_0| \le  \kappa |X|/4$. Therefore, the pair $(X_0, Y_0)$, $X_0 = X \setminus F_0$, is $(3\kappa/4)$-sum-regular. Furthermore, $|X_0| \geq |X| -\kappa |X| /4 \geq 3|X|/4$. 
    Note, again, that $|Y_0| \ge d \geq 2 \sqrt{|X|}$, where the second inequality follows from the assumption of the lemma. 
    
    \paragraph{Phase II.} Set $F, \widehat X, B := \emptyset$. Repeat the following for $\sqrt{|X_0|}$ steps:
    \begin{itemize}
        \item Let $a \in A \setminus (F \cup F_0)$ 
        be the element maximising $|\mS_a(X_0 \setminus \widehat X, Y_0)|$, tie-breaking by taking the smallest one.
        \item Add $a$ to $F$, and for each $x \in \mS_a(X_0 \setminus \widehat X, Y_0)$, add the pair $(x, a + x)$ to $B$.
        \item Set $\widehat X := \left\{ x \in X_0 \colon |B(x, *)| \ge |Y_0| / \sqrt{|X_0|} \right\}$.         
    \end{itemize}
    Note that once an element $x$ becomes a part of $\widehat X$, it stays in $\widehat X$ until the end of the procedure and no more pairs $(x, y)$ are added to $B$. As $|Y_0| \geq d \geq 2 \sqrt{|X|}$, we have, with room to spare,
    \begin{equation} \label{eq:B_x}
        |B(x, *)| \le 2 |Y_0| / \sqrt{|X_0|} \quad \text{ for $x \in X_0$.}
    \end{equation}
    Moreover, since for any $y \in Y_0$ at most one pair $(x, y)$ gets added to $B$ in each iteration, we also have $|B(*, y)| \leq \sqrt{|X_0|}$. We distinguish two cases.
    
    \paragraph{Case 1: $|B| < \kappa \sqrt{|X_0|} \; |Y_0| / 4$.} In this case, set 
    $$ 
        Q: = \{x \in X_0 \setminus F  \colon |\mS_x (X_0 \setminus \widehat X, Y_0)| \geq \kappa |Y_0| / 4\}.
    $$
    Note that $Q$ and $F_0$ are disjoint, because we are considering elements in $X_0$. Furthermore, the upper bound on $|B|$ implies that for every $a \in A \setminus (F \cup F_0)$ we have $|\mS_a(X_0 \setminus \widehat X, Y_0)| < \kappa |Y_0| / 4$. Therefore, $A \cap Q = \emptyset$. From the upper bound on $|B|$ we also get $|\widehat X| < \kappa |X_0| / 4$, thus, using the sum-regularity of $(X_0, Y_0)$, 
    $$
        |\mS(X_0 \setminus \widehat X, Y_0)| \geq |\mS(X_0,Y_0)| - |\widehat X| |Y_0| \geq  \kappa |X_0| |Y_0|/2.
    $$
    Furthermore, since 
    $$  
       \kappa |X_0| |Y_0|/2  < |\mS(X_0 \setminus \widehat X, Y_0)| \leq \left(|X_0| - |F| - |Q| \right) \kappa |Y_0| / 4 + \left(|F| + |Q|\right) |Y_0|,
    $$
    and $|F| = \sqrt{|X_0|} < \kappa |X_0|/8$, we conclude 
    $$
        |Q| \geq \kappa |X_0|/8 \geq \kappa|X|/32.
    $$
    The process stops here as we have found the desired set $Q$. 

    \paragraph{Case 2: $|B| \ge \kappa \sqrt{|X_0|} \; |Y_0| / 4$:} Proceed to the next phase.

    \paragraph{Phase III.} Suppose now that $|B| \ge \kappa \sqrt{|X_0|} \; |Y_0| / 4$. Repeat the following for an additional $\sqrt{|X_0|}$ steps: take $a \in A \setminus (F \cup F_0)$ to be an element which maximises $|B(a, *) \setminus B(F, *)|$, tie-breaking by taking the smallest integer, and add it to $F$. 

    Note that $y \in B(F, *)$ implies $y \in F + F$. By the construction of $B$, we also have $y \in Y_0$. Therefore, $|B(F, *)| \le |(F + F) \cap Y_0|$. As $|F| = 2 \sqrt{|X_0|} \leq 2 \sqrt{|X|}$, and because $F$ did not get removed in phase I, 
    $$
        |B(F, *)| \leq |(F + F) \cap Y_0| < \kappa |Y_0| / 16. 
    $$
    Set $B' = B \setminus (X_0 \times B(F, *))$, and note that for every $a \in A \setminus (F \cup F_0)$ we have 
    $$
        |B'(a, *)| < \frac{\kappa |Y_0|}{16 \sqrt{|X_0|}}.
    $$
    Indeed, if this was not the case then the set $B(F, *)$ would have increased by at least $\kappa |Y_0| / (16 \sqrt{|X_0|})$ in each iteration of Phase III. This would result in $|B(F, *)| \ge \kappa |Y_0| / 16$, which is a contradiction. With this in mind, set
    $$
        Q := \left\{x \in X_0 \colon  |B'(x, *)| \ge \frac{\kappa |Y_0|}{16 \sqrt{|X_0|}} \right\}
    $$
    and note that $A \cap Q = \emptyset$ (by the definition, for $f \in F$ we have $|B'(f, *)| = 0$). As $|B(*, y)| \le \sqrt{|X_0|}$ for each $y \in Y_0$, we have
    $$
        |B'| \ge |B| - |B(F, *)| \sqrt{|X_0|} \ge 3 \kappa \sqrt{|X_0|} \; |Y_0| / 16.
    $$
    From this and \eqref{eq:B_x}, we conclude
    $$
        3 \kappa \sqrt{|X_0|} \; |Y_0| / 16 \le |B'| \le \left(|X_0| - |Q|\right) \frac{\kappa |Y_0|}{16 \sqrt{|X_0|}} +  |Q| \frac{2|Y_0|}{\sqrt{|X_0|}},
    $$
    thus
    $$
        |Q| \ge \kappa |X_0| / 16 \geq \kappa |X| /64.
    $$

    Finally, we set $F = F_0 \cup F$, keep adding the smallest integer in $A \setminus F$ until $|F| = L \log (|Y|/d) \sqrt{|X|}$ and note that by repeating the procedure knowing only the final set $F$, we obtain the same set $Q$.
\end{proof}

The proof of Lemma \ref{lem:missing_prob_sum_regular} is almost identical to the proof of Lemma \ref{lemma:prob_many_missing}, with Lemma \ref{lem:sum_regular_container} taking the role of Lemma \ref{lemma:many_missing_container}.

\begin{proof}[Proof of Lemma \ref{lem:missing_prob_sum_regular}]
    Let $L$ be the constant given by Lemma \ref{lem:sum_regular_container}. By choosing $C$ to be sufficiently large, we may assume $L \log(|Y| / d) \sqrt{|X|} \le |X|$ (otherwise there is no valid choice of $p$).
    
    Set 
    $$
        \mQ = \left\{ (F, Q(F)) : F \in \binom{X}{L \log(|Y|/d) \sqrt{|X|}} \right\}, 
    $$
    where $Q(F)$ is the set $Q$ corresponding to $F$, as given by Lemma \ref{lem:sum_regular_container}. Assumptions of the lemma give $|Y| \geq 2d$ and $|X| \geq C^2/4$, for $C$ large enough compared to $L$, hence we can indeed apply Lemma \ref{lem:sum_regular_container}.
    Moreover, $Q(F)$ is only defined for $F$ corresponding to some set $A \subseteq X$ satisfying $|Y \setminus (A+A)| \geq d$ and $|A| \geq L \log(|Y|/d) \sqrt{|X|}$. For all other sets $F$ we set $Q(F) = X$.

    Lemma \ref{lem:sum_regular_container} implies the following upper bound on the probability of the desired event:
    \begin{equation} \label{eq:A_misses_many_again}
        \Pr\bigg( \big| Y \setminus (A+A) \big| \ge d \bigg) \le \Pr\bigg( |A| \le L \log(|Y|/d) \sqrt{|X|}  \bigg) + \Pr\bigg( \bigcup_{(F, Q) \in \mQ} F \subseteq A \text{ and } Q \cap A = \emptyset \bigg).
    \end{equation}
    
    For $C$ large enough compared to $\kappa$, we have $\mathbb{E}[|A|] = |X|p \ge 2 L \log(|Y|/d)  \sqrt{|X|}$, Chernoff's bound implies
    $$
        \Pr\bigg( |A| \le L \log(|Y|/d) \sqrt{|X|} \bigg) \le \Pr\bigg( |A| \le \mathbb{E}[|A|]/2 \bigg) \le e^{-|X|p/8} < (1 - p)^{|X|/16}.
    $$
    
    The second term in \eqref{eq:A_misses_many_again} is estimated with a union bound:   

    \begin{align*}
        \Pr\bigg( \bigcup_{(F, Q) \in \mathcal{Q}} F \subseteq A \text{ and } Q \cap A = \emptyset \bigg) &\le \sum_{(F, Q) \in \mathcal{Q}} p^{|F|} (1 - p)^{|Q|} \leq \sum_{(F, Q) \in \mathcal{Q}} p^{L \log (|Y|/d) \sqrt{|X|}} (1-p)^{\kappa |X|/64} \\
        &\leq \left( \frac{e |X| p}{L \log(|Y|/d) \sqrt{|X|}} \right)^{L \log(|Y|/d) \sqrt{|X|}} (1-p)^{\kappa |X|/64}.
    \end{align*}
    To ensure enough room sum given \eqref{eq:A_misses_many_again}, we shall show that the above expression is upper bounded by $(1-p)^{\kappa |X|/100}$. Thus, we need to show that
    \[ \left( \frac{e |X| p}{L \log(|Y|/d) \sqrt{|X|}} \right)^{L \log(|Y|/d) \sqrt{|X|}} (1-p)^{ 
  3^2  \kappa |X|/40^2} \leq 1.\]
    This is indeed true, since the left-hand side is decreasing in $p$, and by choosing $C$ large enough for $p = C \log(|Y|/d)/ \sqrt{|X|}$, the result follows. 
    
\end{proof}

\section{Proof of Theorem \ref{thm: missing naturals}}
\label{sec:final}

Instead of showing Theorem \ref{thm: missing naturals} directly, we shall establish its finitary version. This version, adjusted to random $k$-element subsets will play an important role when determining the typical structure of sets with large doubling factor. Furthermore, as we will see shortly, it quickly implies Theorem~\ref{thm: missing naturals}.

\begin{thm} \label{thm:interval_probability}
    There exists $C > 1$ such that the following holds. Let $m,n \in \mathbb{N}$ such that $m \leq 2n/3$, and let $0 < \varepsilon <1/2 $. Let $A \subseteq [ n ]$ be a $p$-random subset, where
    $$
        \frac{C \log (1/\varepsilon)}{\sqrt{\varepsilon^3 m}} \le p \le 1/2. 
    $$
    Then, 
    $$  
        \Pr\bigg( \big| [2, 2n]  \setminus (A+A) \big| \ge m  \bigg) < (1-p)^{\frac{m}{2} - \varepsilon m}. 
    $$
\end{thm}

By the same arguments as given in the introduction, the bound in this finitary statement is also optimal in the sense that the upper bound of the probability is optimal, and so is the regime of $p$ for which it applies, if we consider $\varepsilon$ to be constant. It remains an interesting problem to determine the optimal dependence on~$\varepsilon$. 

For future reference we derive an analogue of Theorem \ref{thm:interval_probability} for random $k$-sets. It is this formulation that will be used to show Theorem \ref{thm: final AP}. The proof of the corollary follows immediately from Pittel's inequality (see, e.g., \cite[Lemma 5.1]{alon2014counting}), thus we omit it.

\begin{cor}
\label{cor: k-set probability}
        There exists $C > 1$ such that the following holds. Let $k,m,n \in \mathbb{N}$ such that $m \leq 2n/3$, and let $0 < \varepsilon < 1/2$. Let $A \subseteq [ n ]$ be a $k$-element subset, chosen uniformly at random among all such sets, where
    $$
        \frac{C \log (1/\varepsilon)}{\sqrt{\varepsilon^3 m}} \le \frac{k}{n} \le 1/2. 
    $$
    Then, 
    $$  
        \Pr\bigg( \big| [2, 2n]  \setminus (A+A) \big| \ge m  \bigg) < (1-k/n)^{\frac{m}{2} - \varepsilon m}. 
    $$
\end{cor}

We now show that Theorem \ref{thm:interval_probability} implies Theorem \ref{thm: missing naturals}.

\begin{proof}[Proof of Theorem \ref{thm: missing naturals}]
    Let $C'$ be the constant given by Theorem \ref{thm:interval_probability}, and set $C = 32C'$.
    Let $A \subseteq \mathbb{N}$ be a $p$-random subset, where $p \ge C \log(1 / \varepsilon) / \sqrt{\varepsilon^3 m}$, and set $n = 2m/p$. Since $p \leq 1/2$, we indeed have $m \leq 2n/3$. Then
    $$
        \Pr \Bigl( | \mathbb{N} \setminus (A+A) | \geq m \Bigr) \leq \Pr \Bigl( | [2, 2n ] \setminus (A+A) | \geq m -1 \Bigr) + \Pr\Bigl( \mathbb{N} \setminus [ 2n ] \nsubseteq A + A\Bigr).
    $$
    The lower bound on $p$ allows us to apply Theorem \ref{thm:interval_probability} on $A' = A \cap [n]$ with $\varepsilon/3$ (as $\varepsilon$) and $m -1$ (as $m$). We obtain
    \[ \Pr \Bigl( |[2,2n] \setminus (A + A)| \ge m-1 \Bigl) \le \Pr \Bigl( | [ 2, 2n ] \setminus (A'+A') | \geq m -1 \Bigr) < (1-p)^{\frac{m}{2} - \varepsilon m / 2}, \]
    where $m/2 - \varepsilon m/2 < (m-1)/2 - \varepsilon(m-1)/3$ followed from $\sqrt{\varepsilon^3 m} \geq 2C \log (1/\varepsilon)$.
    For $x \in \mathbb{N}$, let $\mathcal{E}_x$ denote the event $x \not \in A + A$. Then we claim
    \[  \Pr(\mathcal{E}_x) \le \Pr\left( \bigcap_{a = 1}^{\lfloor x/2 \rfloor} (a \not \in A) \cup (x - a \not \in A) \right) \le (1 - p^2)^{(x-1)/2}.\]
    Indeed, if $x$ is odd this is immediate, and when $x$ is even we have for $0 \leq p \leq 1$
    \[ \Pr\left( \bigcap_{a = 1}^{\lfloor x/2 \rfloor} (a \not \in A) \cup (x - a \not \in A) \right) = (1-p^2)^{(x-2)/2} \cdot (1-p) \leq (1-p^2)^{(x-1)/2}.  \]
    Furthermore, using $1 - p^2 \le (1 - p)^{p/2}$ for $0 \le p \le 1$, it follows that
    \begin{equation} \label{eq:bound_Ex}
        \Pr(\mathcal{E}_x) \leq (1 - p)^{p(x-1)/4}.
    \end{equation}
    By a union bound,
    \begin{align*}
        \Pr \Bigl( \mathbb{N} \setminus [ 2n ] \nsubseteq A+A \Bigr)& = \Pr\bigg(\bigcup_{x = 2n+1}^\infty \mathcal{E}_x \bigg) \le \sum_{x = 2n +1}^\infty \Pr ( \mathcal{E}_x ) \\
        &\leq \sum_{x=2n+1}^\infty (1 - p)^{p(x-1)/4} = \frac{(1 - p)^{np/2}}{1 - (1-p)^{p/4}} \le \frac{4}{p^2} (1 - p)^{m} < (1 - p)^{m/2}.
    \end{align*}
    The penultimate inequality follows from $n = 2m/p$ and the fact that $p^2/4 \leq 1 - (1-p)^{p/4}$, for $0 \leq p \leq 1$. To see why the last inequality holds, note that for $C \geq 4$, we have $5 m e^{-C \sqrt{m}/2} \leq C^2$ for any $m >0$. Thus, with a suitable choice for $C$, $(1-p)^{m/2} \leq e^{-pm/2} \leq e^{-C \sqrt{m}/2} \leq C^2/(5m) \leq p^2/5$.
    This implies the theorem.
\end{proof}

\begin{proof}[Proof of Theorem \ref{thm:interval_probability}] 
    Suppose $|[2,2n] \setminus (A + A)| \ge m$. Let $M = (2^{11} + K_0/2) m$, where $K_0$ is the constant given by Lemma \ref{lemma:missing_few_elements}. We first deal with the case where $M \ge n$. Then, we apply Lemma \ref{lemma:prob_many_missing} for $X = [n] = [Tm]$ for $T < 2^{11} + K_0/2$ and $\varepsilon/2$. Note that the application of this lemma requires $m \leq 2n/3$. Choosing $C$ large enough compared to the constant given by Lemma \ref{lemma:prob_many_missing} and compared to $T$ indeed allows us to do so, and obtain 
    \[ \Pr( |[2, 2n] \setminus (A+A)| \geq m ) \leq (1-p)^{m/2 - \varepsilon m}. \]

    Suppose now that $M < n$. By the pigeonhole principle, we have
    \begin{itemize}
        \item $| [2,M] \setminus (A+A)| \ge d_1$ and $|[2n - M + 1, 2n] \setminus (A+A)| \ge d_2$, for some $d_1, d_2 \in \mathbb{Z}_{\geq 0}$ such that $d_1 + d_2 = m - \varepsilon m / 2$, or
        \item $|[M + 1, 2n - M ] \setminus (A + A)| \ge \varepsilon m / 2$.
    \end{itemize}
    To bound the probability of the latter event, we apply Lemma \ref{lemma:missing_few_elements} with $M$, $d = \varepsilon m/2$ and hence $K = (2^{12} + K_0)/\varepsilon \geq K_0$. We are indeed allowed to use the lemma, as we satisfy the three required conditions:
    \begin{itemize}
        \item The condition $d/\log^3 d \geq 2^8 K$, corresponds to 
        \[ \frac{\varepsilon^2 m/2}{\log^3 (\varepsilon m/2)} \geq 2^{20} + 2^8 K_0. \]
        Note that since $C\log(1/\varepsilon)/\sqrt{\varepsilon^3 m} \leq 1$, it suffices to choose $C$ large enough with respect to $K_0$ to ensure that the inequality holds.
        \item Secondly we need 
        \[ p \geq \frac{(2^{12} + K_0) \log((2^{12} + K_0)\cdot \varepsilon^{-1})}{\varepsilon \sqrt{(2^{11} + K_0/2) \cdot  m}}. \]
        This follows from $p \geq C \log(1/\varepsilon) /\sqrt{\varepsilon^3 m}$, and letting $C$ be large enough compared to $K_0$.
        \item The final condition comes down to 
        \[ p \geq \frac{(2^{12} + K_0)^2 \log^2(\varepsilon m/2)}{\varepsilon^3 m/2}. \]
        By the same reasoning as above, this also holds true.
    \end{itemize}
    Thus, the probability of the latter event is at most $(1 - p)^{M / 2^{11}} \le (1 - p)^m$. 
    
    We now estimate the probability of the former event. This is done by a union bound over $d_1, d_2 \in \mathbb{Z}_{\geq 0}$ with stated properties. 
    Consider one such pair of integers $d_1$ and $d_2$, and set $X_1 = [ M/2 ]$ and $X_2 = [n - M/2, n]$, of respective lengths $T_i d_i$.
    If $d_i = 0$, we trivially have 
    \[ \Pr\bigg( |(X_i + X_i) \setminus (A_i + A_i)| \geq d_i \bigg) \le 1.\] 
    For non-zero $d_i$, set $\varepsilon_i = \varepsilon m / (16 d_i)$ and note $T_i \leq (M+2) / (2d_i)$. Let $A_i = A \cap X_i$, and note that $A_i$ is a $p$-random subset of $X_i$. If $d_i \ge \varepsilon m / 4$, then by Lemma \ref{lemma:prob_many_missing} we have
    $$
        \Pr\bigg( |(X_i + X_i) \setminus (A_i + A_i)| \geq d_i \bigg) \le (1 - p)^{d_i / 2 - 2 \varepsilon_i d_i} = (1 - p)^{d_i / 2 -  \varepsilon m/8}.
    $$
    We are indeed able to use Lemma \ref{lemma:prob_many_missing} because of the assumption $p \geq C \log(\varepsilon^{-1})/\sqrt{\varepsilon^3 m}$. If $d_i < \varepsilon m / 4$, then we trivially have
    $$    
        \Pr\bigg( |(X_i + X_i) \setminus (A_i + A_i)| \geq d_i \bigg) \le (1 - p)^{d_i / 2 - \varepsilon m/8}.
    $$
    In either case, as $X_1 \cap X_2 = \emptyset$ the events $|(X_1 + X_1) \setminus (A_1 + A_1)| \ge d_1$ and $|(X_2 + X_2) \setminus (A_2 + A_2)| \ge d_2$ are independent, thus
    $$
        \Pr\bigg( |(X_1 + X_1) \setminus (A + A)| \ge d_1 \text{ and } |(X_2 + X_2) \setminus (A + A)| \ge d_2 \bigg) \le (1 - p)^{m/2 - \varepsilon m / 2}.
    $$ 

    Altogether, we have
    $$
        \Pr\bigg( \big| [2, 2n] \setminus (A + A) \big| \ge m \bigg) \le (1 - p)^m + m (1 - p)^{m/2 - \varepsilon m / 2} \le (1 - p)^{m/2 - \varepsilon m}.
    $$    
\end{proof}

\section{Structural containers}
\label{sec: structural containers}
In this section, we develop the second major tool required to show Theorem \ref{thm: final AP}. This tool, also founded on sum-robust sets, is a container lemma that besides containing sets with bounded doubling factor, also specifies the structure of the containers. The main improvement over the results from \cite{campos2020number,liu2024number} comes from the dependence of the size of the family of containers on the parameter $\varepsilon$.

\begin{lem}[Structural Container Lemma] \label{lem:container_approx}
    There exist $t_0, L \in \mathbb{N}$ such that the following holds. Let $n, t \in \mathbb{N}$ with $t \ge t_0$, and $L t^{-1/3} < \varepsilon < 2^{-12}$. Then there exists a family $\mC$ of subsets of $\mathbb{N}$ with the following properties: 
    \begin{itemize}
        \item For every $A \subseteq [ n ]$ such that $|A+A| \le t$, there exists $S \in \mC$ such that $A \subseteq S$,
        \item $|\mC| < n^{L \sqrt{t/\varepsilon}}$, for some absolute constant $L > 1$, and
        \item for every $S \in \mC$, either $|S| \le (1 - \varepsilon)t/2$, or there exists an arithmetic progression $P \subseteq S$ of size $t/2$ such that $|S \setminus P| \leq 2^{10} \varepsilon t$.
    \end{itemize}
\end{lem}

This proof of the Structural Container Lemma makes use of sum-robustness, and proceeds very similarly as Section \ref{sec: missing many elements}, by first establishing two results showing sum-robustness of pairs of sets, and then by showing how to find the containers. However, there are some obvious differences: the set we consider is no longer an interval, and we are aiming to conclude that $S$ is either very small, or the reason it cannot be shrunk any more is that it is contained in a long arithmetic progression.

\begin{lem}
\label{lem: step 1 robust}
Suppose $0 \leq \delta \leq 3/2,  0 \leq \zeta \leq 1$, and $t \in \mathbb{N}$, and let $X, Y \subseteq \mathbb{N}$ be sets of size 
$$
    |X| = (1/2 + \delta) t \quad \text{ and } \quad |Y| = (1-\zeta) t.
$$
If $\max (\delta, \zeta) \cdot t \geq 8$, then $(X, \overline{Y})$ is $\beta$-sum-robust for $\beta = \max (\delta, \zeta) / 16$.
\end{lem}
\begin{proof}
    Consider some $R_X \subseteq X$ and $R_Y \subseteq \overline{Y}
    $ with $|R_X|, |R_Y| \le \beta |X|$. Set $B = X \setminus R_X$ and $A = Y \cup R_Y$, thus $|B| \geq (1-\beta) |X|$ and $|A| \leq |Y| + \beta |X|$. For every $(x, x') \in B \times B$ such that $x + x' \not \in A$ we have $x + x' \in \overline{Y} \setminus R_Y$, paving the way for Pollard's Theorem we first show $|B| \ge (1/2 + 2 \beta)(|A|+1)$. That is, the following inequality has to be satisfied:
    \[ \left( (1-\beta) - (1/2+ 2\beta) \beta  \right) |X| \geq (1/2 + 2\beta)( |Y| +1). \]
    Since $\beta \leq 1/4$, we simplify this to
    \[ (1- 2 \beta) |X| \geq (1/2 + 2 \beta) (|Y|+1). \]
    Since $|X| \geq t/2$, $|Y| \leq t$, and $\beta \leq 3/32$, this is indeed satisfied whenever $\max(\delta, \zeta) \cdot t \geq 8$.
    So by Pollard's Theorem, 
    \[| \mS (X \setminus R_X, \overline{Y} \setminus R_Y)| \geq 2 \beta^2 |B|^2 \geq \beta^2 |X|^2,\] 
    since $(1-\beta)^2 \geq 1/2$.
\end{proof}

\begin{lem}
\label{lem: step 2 cases}
   Suppose $0 < \varepsilon < 2^{-12}$ and $t \in \mathbb{N}$ such that $\varepsilon t \geq 1/2$, and let $X, Y \subseteq \mathbb{N}$ be sets of size 
    $$
        (1-2\varepsilon) t/2 \leq |X| \leq (1/2 +  \varepsilon) t \quad \text{ and } \quad |Y| = (1-\zeta) t,
    $$
    for 
    $0 \leq \zeta \leq 1$.
    Then at least one of the following is true:
    \begin{enumerate}[({B}1)]
    \item The pair $(X, \overline{Y})$ is $\beta$-sum-robust for $\beta = \max(\varepsilon, \zeta/16)$; or
    \item \label{prop:H_eps_robust} There exists an arithmetic progression $P$ of length $|P| \leq t/2$ such that $|X \setminus P| \leq 2^{10} \varepsilon t$.
    \end{enumerate}
\end{lem}

\begin{proof}
    Note that if $|X| \geq t/2$ and $\zeta/16 \geq \varepsilon$, then we have $\zeta t \geq 8$ because $\varepsilon t \geq 1/2$. Thus the result follows from Lemma \ref{lem: step 1 robust}.

    If $(1-2 \varepsilon)t/2 \leq |X| < t/2$ and $\zeta/16 \geq \varepsilon$, then by the same argument as the proof of Lemma \ref{lem: step 1 robust} it suffices to show that 
    \[ (1-2 \beta) |X| \geq (1/2+ 2 \beta) (|Y|+1). \]
    This follows from plugging in the bound $|X| \geq (1-\zeta/8) t/2$, so that the bound holds whenever
    \[ \frac{\zeta t}{16} \left( \frac{17 \zeta}{8} +4 \right) \geq \frac{5}{8}. \]
    This is indeed satisfied since $\varepsilon t \geq 1/2$.
    
    Lastly, we assume that $(1-2\varepsilon) t/2 \leq |X| \leq (1/2 + 2 \varepsilon) t $ and that $\zeta < 16 \varepsilon$. In this case, rather than using Pollard's Theorem, we shall use Lemma \ref{lemma: AP in Pollard} with $4 \varepsilon$ (as $\varepsilon$). Note that $ \varepsilon < 2^{-12}$, which allows us to use that lemma. Consider $R_X \subseteq X$ and $R_Y \subseteq \overline{Y}$, both of size $|R_X|, |R_Y| \leq \varepsilon |X|$. Set $B = X \setminus R_X$ and $A =Y \cup R_Y$, such that $|B| = (1-\varepsilon) |X|$ and $|A| \leq t + \varepsilon |X|$. We shall show that 
    \[ (1 - 4\varepsilon) (t + \varepsilon |X|)/2 \leq |B| \leq (1/2 + 4\varepsilon) (t + \varepsilon |X|), \]
    and that knowing this, the result follows easily from the lemma. Furthermore, observe that the upper bound follows easily from $|X| < (1/2 + \varepsilon) t$. To see that the lower bound also holds, note that 
    \[ 1 - 4\varepsilon \leq \left(1 - \varepsilon - ( 1- 4 \varepsilon) \; \frac{\varepsilon}{2}\right) (1-2 \varepsilon), \]
    for $\varepsilon$ small. Thus,
    \[ (1-4 \varepsilon)t/2 \leq \left(1- \varepsilon - (1-4 \varepsilon)\; \frac{\varepsilon}{2}\right) |X|. \]
    So by Lemma \ref{lemma: AP in Pollard}, either
    \begin{enumerate}[({A}1)]
        \item $|\mS (B, A )| \geq 2 (4\varepsilon)^2 (t + \varepsilon |X|)^2 \geq \varepsilon^2 |X|^2 $, or
        \item There exists an arithmetic progression $P'$, $|P'| \leq t/2 + \varepsilon |X|/2 \leq (t+ \varepsilon t)/2$ such that $|B \setminus P'| \leq 2^8 \varepsilon (t + \varepsilon |X|) \leq 2^8 \varepsilon t + 2^8 \varepsilon^2 t $. Consequently, there exists an arithmetic progression $P$ of length $|P| \leq t/2 $ and $|X \setminus P| \leq 2^{10} \varepsilon t$.  
    \end{enumerate}
\end{proof}

We conclude this section with the proof of Lemma \ref{lem:container_approx}. To obtain the desired containers, we shall develop a 3-stage procedure: at first we do a preprocessing step. Then, while the containers are still large enough (i.e. larger than $(1/2 + \varepsilon)t$) Pollard's Theorem will allow us to easily refine them. Finally, we use Lemma \ref{lemma: AP in Pollard} to push the containers even further.

\begin{proof}[Proof of Lemma \ref{lem:container_approx}]
    We prove that for each $A \subseteq [ n ]$ with $|A + A| \le t$, one can find a subset $F = F(A) \subseteq A$ of size $|F| \le \frac{L}{2} \sqrt{t / \varepsilon}$, and a set $S = S(F) \subseteq \mathbb{N}$ such that $A \subseteq S$ and either $|S| \le (1 - \varepsilon)t/2$ or there exists an arithmetic progression $P \subseteq S$ of size $t/2$ such that $|S \setminus P| < 2^{10} \varepsilon t$. Crucially, $S(F)$ only depends on $F$. The family $\mC$ is then simply defined as 
    $$
        \mC := \{ S(F(A)) : A \subseteq [ n ], |A + A| \le t\}.
    $$
    The size of the family $\mC$ is bounded by
    $$
        |\mC| \le \sum_{i = 0}^{\frac{L}{2} \sqrt{t / \varepsilon}} \binom{n}{i} < n^{ L \sqrt{t / \varepsilon}},
    $$
    with room to spare.
     Define $L$ as
    $$
        L := 12 + 32 \sum_{i = 0}^\infty (1 - 2^{-12})^{i/2} + 32 \sum_{j = 0}^\infty (1-2^{-13})^{j/2}.
    $$
    This choice of $L$ will only be relevant in \eqref{eq:F_size}. 
    
    Consider some $A \subseteq [n]$ such that $|A+A| \le t$. If $|A| < \frac{L}{2} \sqrt{t/\varepsilon}$, then we can simply set $F := A$ and $S := F = A$. Since $\frac{L}{2}\sqrt{t / \varepsilon} < t/2$ (which holds as $t \ge t_0$), the conclusion of the lemma is satisfied. For the rest of the proof, suppose $|A| \geq \frac{L}{2} \sqrt{t / \varepsilon}$. 

    \paragraph{Initial container.} Let $F_0 \subseteq A$ be the set consisting of $\sqrt{t}$ smallest elements in $A$. As long as there exists $a \in A$ such that $|(a + F_0) \setminus (F_0 + F_0)| \geq \sqrt{t}/5$, take the smallest $a \in A$ which maximises $|(a + F_0) \setminus (F_0 + F_0)|$ and add it to $F_0$. Now, define the initial container $C_0$: 
    $$
        C_0 = \{ x \in [n]: |(x+ F_0) \setminus (F_0 + F_0)| < \sqrt{t}/5 \}.
    $$
    At the end of this step, we claim: 
    \begin{itemize}
        \item $|F_0| \le 6 \sqrt{t}$: Before we started adding elements to $F_0$, we clearly had $|F_0 + F_0| \geq \sqrt{t}$. Moreover, at every step we added an element to $F_0$, the size of the set $F_0 + F_0$ increased by at least $\sqrt{t}/5$. Since $F_0 + F_0 \subseteq A + A$ and $|A + A| \leq t$, we conclude $|F_0| \leq 6 \sqrt{t}$.

        \item $|C_0| \leq 2t$: Suppose, towards a contradiction, that $|C_0| > 2t$. By Lemma \ref{lem: asymmetric pollard} applied with $\varepsilon$ being arbitrarily close to $1/2$, we have $ |C_0| + |F_0| > 2t + \sqrt{t} \geq (1+ 2\varepsilon) (|F_0 + F_0| +1)$.
    Hence there are at least $\varepsilon^2 |C_0| |F_0| \geq |C_0| \sqrt{t}/5$ pairs $(x,f) \in C_0 \times F_0$ such that $x + f \notin F_0 + F_0$. Thus, there is an $x \in C_0$ with the property that $|(x + F_0) \setminus (F_0+F_0) | \geq \sqrt{t}/5$, 
    which contradicts the definition of $C_0$.
    \end{itemize}
    Set $i = 0$, $Q_0, W_0 = \emptyset$, and consider the following two phases.
    
    \paragraph{Phase 1.} Repeat the following until $|C_i| \le (1/2 + \varepsilon)t$:
    \begin{itemize}
        \item Apply Lemma \ref{lem:robust_pair_container} on $(C_i, (C_i + C_i) \setminus W_i)$ with 
        $$
            \beta_i = \max(\delta_i, \zeta_i)/16, 
        $$
        where $\delta_i, \zeta_i$ are such that $|C_i| = (1/2 + \delta_i)t$ and $|W_i| = (1- \zeta_i) t$. Let $F_{i+1} \subseteq A$, $|F_{i+1}| \le 2 \sqrt{|C_i|/\beta_i}$, and $Q_{i+1}(F_{i+1}) \subseteq C_i$ be the obtained sets.

        \item Set $C_{i+1} = C_i \setminus Q_{i+1}$ and $W_{i+1} = W_i \cup (F_{i+1} + F_{i+1})$. Increase $i$, and proceed to the next round.
    \end{itemize}

    \paragraph{Phase 2.} Repeat the following until $|C_i| < (1/2 - \varepsilon)t$ or there exists an arithmetic progression $P$ of length $|P| \le t/2$ such that $|C_i \setminus P| \le 2^{10} \varepsilon t$:
    \begin{itemize}
          \item Apply Lemma \ref{lem:robust_pair_container} on $(C_i, (C_i + C_i) \setminus W_i )$ with 
          $$
            \beta_i = \max(\varepsilon, \zeta_i/16),
          $$
          where $|W_i| = (1- \zeta_i)t$. Let $F_{i+1} \subseteq A$, $|F_{i+1}| \le 2 \sqrt{|C_i|/\beta_i}$, and $Q_{i+1}(F_{i+1}) \subseteq C_i$ be the obtained sets.

        \item Set $C_{i+1} = C_i \setminus Q_{i+1}$ and $W_{i+1} = W_i \cup (F_{i+1} + F_{i+1})$. Increase $i$, and proceed to the next round.
    \end{itemize}
    Once the procedure has terminated, which we shall verify shortly, set $F = \bigcup_{i = 0}^{\ell} F_i$ where $\ell$ denotes the final value of $i$ in the procedure. While $|F| < \frac{L}{2} \sqrt{t / \varepsilon}$, add to $F$ the smallest integer in $A \setminus F$.  Set $S = C_\ell$. If the procedure has finished because of the existence of an arithmetic progression $P$, then extend $P$ by adding larger elements to obtain $\widehat{P}$, with $|\widehat{P}| = t/2$, and add $\widehat{P}$ to $S$. Clearly, we have $|S \setminus \widehat{P}| \leq |C_\ell \setminus P| \leq 2^{10} \varepsilon t $. 

    \paragraph{The procedure is well-defined.} The only thing we need to verify is that we can indeed apply Lemma~\ref{lem:robust_pair_container} in every iteration. There are three conditions to verify:
    \begin{itemize}
        \item $(C_i, \overline{W_i})$ is $\beta_i$-sum-robust: Since $\varepsilon t > L$, this follows immediately from $\mS (C_i,  (C_i + C_i) \setminus W_i) = \mS (C_i, \overline{W_i})$ and Lemmas~\ref{lem: step 1 robust} and \ref{lem: step 2 cases}; 
        
        \item $\sqrt{|C_i|/\beta_i} < \beta_i |C_i|/4$: Straightforward calculation shows that this is equivalent to $\beta_i > (16 /|C_i|)^{1/3}$. This is satisfied as
        \[ \beta_i > \frac{\varepsilon}{16} > \frac{L \; t^{-1/3}}{16} > \frac{ L \; (t/3)^{-1/3}}{32} > (16 /|C_i|)^{1/3}, \]
        since $L > 2^7$.
        
        \item $|A| \ge 2 \sqrt{|C_i|/\beta_i}$: 
        \[ |A| \geq \frac{L}{2} \sqrt{\frac{t}{\varepsilon}} > \frac{L}{32} \; \sqrt{\frac{2t}{\varepsilon/16}} > 2 \sqrt{\frac{|C_i|}{\beta_i}}. \]
    \end{itemize}

    \paragraph{The procedure eventually terminates.} As $\beta_i \ge \varepsilon / 16$, Lemma \ref{lem:robust_pair_container} guarantees 
    $$
        |(F_{i+1} + F_{i+1}) \cap \overline{W_i}| + |Q_{i+1}| \ge 2^{-10} \varepsilon |C_i| > 2^{-12} \varepsilon t.
    $$
    Therefore, $(t - |W_i|) + |C_i|$ decreases by at least $2^{-12} \varepsilon t$ in each iteration and since $W_\ell \subseteq A + A$ this eventually needs to stop.

    \paragraph{Bound on the size of $F$.} Let $\ell_1$ denote the number of iterations in phase 1. We have
    $$
        \sum_{i = 1}^{\ell_1} |F_{i}| \le \sum_{i = 1}^{\ell_1} 2 \sqrt{\frac{|C_{i-1}|}{\beta_{i-1}}} \le \sqrt{8 t} \; \sum_{i = 1}^{\ell_1 } \beta_{i-1}^{-1/2} \le 16 \sqrt{t} \; \sum_{i = 0}^{\ell_1 -1} (\delta_{i} + \zeta_{i})^{-1/2},
    $$
    where we have used that $|C_i| \le 2t$ and $\beta_i \ge (\delta_i + \zeta_i) / 2^5$. Note that
    \begin{equation} \label{eq:delta_zeta}
        \delta_i + \zeta_i \le (1 - 2^{-12}) (\delta_{i-1} + \zeta_{i-1})
    \end{equation}
    for every $1 \le i \leq \ell_1$. This can be seen as follows:
    \begin{align*}
        (-1/2 + \delta_i + \zeta_i) t = |C_i| - |W_i| &= |C_{i-1}| - |Q_{i}| - |W_{i-1}| - \left|\overline{W_{i-1}} \cap (F_{i} + F_{i})\right| \\
        &\le |C_{i-1}| - |W_{i-1}| - \beta_{i-1} |C_{i-1}|/64 \\
        &\le (-1/2 + \delta_{i-1} + \zeta_{i-1})t - \frac{\delta_{i-1} + \zeta_{i-1}}{2^{12}} t.
    \end{align*}
    As $\delta_{\ell_1-1} + \zeta_{\ell_1 - 1} \ge \delta_{\ell_1 - 1} \ge \varepsilon$,  \eqref{eq:delta_zeta} further implies
    $$
        \delta_{\ell_1 - 1 - j} + \zeta_{\ell_1 - 1 - j} \ge (1 - 2^{-12})^{-j} \varepsilon.
    $$
    Thus the contribution to $F$ in the first phase is bounded by
    \begin{equation}
    \sum_{i = 1}^{\ell_1 } |F_i| \le 16 \sqrt{t} \; \sum_{i=0}^{\ell_1-1} (\delta_{\ell_1 - 1 - i} + \zeta_{\ell_1 - 1 - i})^{-1/2} \leq 16 \sqrt{t/\varepsilon} \; \sum_{i=0}^{\ell_1-1} (1-2^{-12})^{i/2} .
    \end{equation}
   Now let $\ell_2$ denote the number of iterations in phase 2. Let $I = \{i_0, i_1, \ldots, i_{\ell'} \} \subseteq \{1, \ldots, \ell_2\}$ denote the set of all indices such that $|W_{i-1}| \leq (1- 16\varepsilon) t$ and $|Q_{i}| \leq \beta_{i-1} |C_{i-1}| / 128$. Then
    $$
        \sum_{i \in I} |F_i| \le \sum_{i \in I} 2 \sqrt{\frac{|C_{i-1}|}{\beta_{i-1}}} \le 8 \sqrt{t} \; \sum_{i \in I} \zeta_{i-1}^{-1/2}.
    $$
    Let $i \in I$, then we have
    \[ ( 1- \zeta_{i}) t  = |W_{i}| \geq |W_{i-1}| + \beta_{i-1} |C_{i-1}| /128 \geq   |W_{i-1}| + \frac{\zeta_{i-1} t}{2^{13}} \geq \left(1- \left(1-2^{-13}\right) \zeta_{i-1}\right) t, \]
    and thus
    \[ \zeta_{i} \leq (1 - 2^{-13}) \zeta_{i-1}.  \]
    From $\zeta_{i-1} \ge \varepsilon$, again by the definition of $I$ and by monotonicity, for each $j \in \{0, \ldots, \ell'\}$ we have
    $$
        \zeta_{i_{\ell' - j}} \ge (1 - 2^{-13})^{-j} \varepsilon. 
    $$
    Therefore, 
      \[  \sum_{i \in I} |F_i| \le 8 \sqrt{t/\varepsilon} \; \sum_{j = 0}^{\infty} (1 - 2^{-13})^{j/2}. \]
    
    Furthermore, since $|C_{\ell_1}| \leq (1/2 + \varepsilon)t$ and the process finishes whenever $|C_i| \leq (1/2 - \varepsilon)t$ there can be at most $2^{11}$ rounds such that $|Q_i| \geq \beta |C_i|/128 \geq \varepsilon t/2^{10}$.  
    And if $|W_i| \geq (1-16\varepsilon)t$ and $|Q_i| \leq \beta |C_i|/128$, then $|W_{i+1}| \geq |W_i| + \varepsilon t/2^{10}$, and thus there at most $2^{14}$ such rounds. Therefore, 
    \[ \sum_{i \not \in I} |F_i|  =  \sum_{i \not \in I} 2 \sqrt{|C_{i-1}|/\beta_{i-1}} \leq 2 \cdot (2^{11} + 2^{14}) \sqrt{t/\varepsilon} \leq 8 \sqrt{t/\varepsilon} \; \sum_{j = 0}^{\infty} (1 - 2^{-13})^{j/2}.  \]
    Combining the two phases, we find 
    \begin{equation} \label{eq:F_size}
        |F| \leq 6 \sqrt{t/\varepsilon} + 16 \sqrt{t/\varepsilon} \; \sum_{i=0}^{\infty} (1-2^{-12})^{i/2} + 16 \sqrt{t/\varepsilon} \; \sum_{j = 0}^{\infty} (1 - 2^{-13})^{j/2} = \frac{L}{2} \sqrt{t/\varepsilon}.
    \end{equation}

    \paragraph{Computing $F$ from $A' \subseteq A$.} If $|A| \le \frac{L}{2} \sqrt{t / \varepsilon}$ then necessarily $F = A$, thus $F \subseteq A' \subseteq A$ implies $A = A'$. Consider now the more interesting case $|A| > \frac{L}{2} \sqrt{t / \varepsilon}$, and suppose $F \subseteq A' \subseteq A$. Note that $|F| = \frac{L}{2} \sqrt{t / \varepsilon}$ in this case. Since $F_0 \subseteq F$, by running the initial step for obtaining $F_0$ on $A'$ instead of $A$, we select the same set of vertices. Therefore, $F_0$ can be retrieved from $A'$. Let $F_i', C_i', W_i'$ be the sets obtained by running the subsequent phase 1 with $A'$ instead of $A$. Using the analogous property of Lemma \ref{lem:robust_pair_container}, we inductively have that $F_i' = F_i$ for each $i \ge 0$, and thus $C_i' = C_i$ and $W_i' = W_i$. This means we produce the same sequence of $F_i$'s, in Phase 1, and the analysis for Phase 2 is exactly the same. Thus we obtain the same set $F$.
    
\end{proof}

\section{The typical structure of sets with large sumset}
\label{sec: typical structure}

Using the probability of missing elements, Corollary~\ref{cor: k-set probability}, and the Structural Container Lemma, we finally derive Theorem~\ref{thm: final AP}. 

\begin{proof}[Proof of Theorem \ref{thm: final AP}]
    Note that we can trivially assume that $n > \lambda k/2$. Let $L$ be the absolute constant given by Lemma \ref{lem:container_approx}. Since  $\lambda = o(k^{1/4} /\log^{1/2} n)$ and $k \geq \log^2 n$, we can assume
    $$ 
        \frac{\lambda}{k} \log^2 n < 2^{-50} L^{-2}.
    $$
    Furthermore, we also have
    $$
        \varepsilon := 2 L^{2/3} \left( \frac{\lambda}{k} \log^2 n \right)^{1/3} < 2^{-12}.  
    $$
    This particular choice of $\varepsilon$ is made with the following inequality in mind:
     \begin{equation} \label{eq:eps_choice}
        \varepsilon k > 2 L \sqrt{\frac{\lambda k}{\varepsilon}} \log n.
    \end{equation}
    Furthermore, note that 
    \[ \frac{L}{\lambda^{1/3}k^{1/3}} \leq \varepsilon. \]
    Thus we are able to apply Lemma \ref{lem:container_approx} for $t = \lambda k $ and $\varepsilon$; let $\mC$ denote the family of sets given, and set
    $$
        \mC_s = \left\{ S \in \mC \; : \; |S| \le (1 - \varepsilon)\lambda k / 2 \right\} \quad \text{ and } \quad \mC_p = \mC \setminus \mC_s.
    $$
    Let 
    $$
        \mF = \{A \subseteq [n] \; : \; |A| = k \text{ and } |A + A| \leq \lambda k\}
    $$
    denote the collection of all admissible sets $A$, and set
    $$
        \mF_s = \left\{ A \in \mF \; : \; A \subseteq S \text{ for some } S \in \mC_s \right\} \quad \text{ and } \quad \mF_p = \mF \setminus \mF_s.
    $$
    Let 
    $$
        \gamma :=  C_\gamma \lambda^{4/3} k^{-1/3} \log^{2/3} n,
    $$
    where $C_\gamma$ is sufficiently large, and
    $$
        \mG = \{ A \in \mF: \;\exists \text{ arithmetic progression } P \text{ such that } |P| = \frac{\lambda k}{2} + 2 \gamma \lambda k \text{ and } A \subseteq P \}. 
    $$
    Note that $\gamma \to 0$ as $n \to \infty$, and that Theorem \ref{thm: final AP} states $|\mF \setminus \mG|  = o_{n \to \infty} \left( \binom{\lambda k / 2}{k} \right)$. Since
    $$
        |\mF_s| \leq |\mC_s| \binom{(1 - \varepsilon) \frac{\lambda k}{2}}{k} \stackrel{\eqref{eq:a-c_b-d}}{\leq} |\mC_s|  2^{- \varepsilon k} \binom{\lambda k/2}{k} \stackrel{|\mC_s| \le |\mC|}{\le}   2^{L \sqrt{\frac{\lambda k}{\varepsilon}} \log n - \varepsilon k} \binom{\lambda k/2}{k} \stackrel{\eqref{eq:eps_choice}}{=} o_{n \to \infty}\left( \binom{\lambda k / 2}{k} \right).
    $$
    it suffices to consider $\mF_p \setminus \mG$. 

    For each $S \in \mC_p$, let $P_S \subseteq S$ be an arithmetic progression of size $\lambda k / 2$ such that $|S \setminus P| \leq 2^{10} \varepsilon \lambda k$. Such an arithmetic progression is  guaranteed by Lemma \ref{lem:container_approx}. For each $A \in \mF_p$, let $S_A \in \mC_p$ be a set such that $A \subseteq S_A$, and set $P_A := P_{S_A}$. We first show that the vast majority of sets $A \in \mF_p$ are sufficiently well-distributed across $P_{A}$. Let $P_A^-$ denote the smallest $\gamma \lambda k$ and $P_A^+$ the largest $\gamma \lambda k$ elements in $P_A$, and set
    $$
        \mB = \left\{ A \in \mF_p \; : |A \cap P_A^-| \leq \gamma k / 8 +1 \text{ or } |A \cap P_A^+| \leq \gamma k / 8 +1 \right\}. 
    $$
    The following claim is a simple application of a Chernoff bound coupled with elementary counting.
    \begin{claim}
        $|\mB| = o_{n \to \infty}\left( \binom{\lambda k / 2}{k} \right)$.
    \end{claim}
    \begin{proof}
        For each $S \in \mC_p$, let
        $$
            \mB^+_S = \left\{A \subseteq S \; : \; |A| = k \text{ and } |A \cap P_S^+| \le \gamma k / 8 + 1 \right\}.
        $$
        The family of `bad' subsets $\mB_S^-$ with respect to $|A \cap P_S^-|$ is defined analogously. Then
        $$
            \mB \subseteq \bigcup_{S \in \mC_p} \mB^+_S \cup \mB^-_S.
        $$
        Sets $\mB_S^+$ correspond to the event that a $k$-subset $A \subseteq S$, chosen uniformly at random, contains at most $\gamma k / 8$ elements from $P_S^+$. For such a randomly chosen $A$, let $X := |A \cap P_S^+|$. Then $X \sim \operatorname{Hypergeometric}(|S|, |P_S^+|, k)$ and
        $$
            \mathbb{E}[X] = \frac{k |P_S^+|}{|S|} \ge \frac{k \cdot \gamma \lambda k}{\lambda k} = \gamma k.
        $$
        The standard lower-tail bound for a hypergeometric random variable gives
        $$
            \Pr\left( X \le \mathbb{E}[X] / 2 \right) \le e^{-\mathbb{E}[X] / 8},
        $$
        thus 
        $$
            |\mB^+_S| \le e^{-\gamma k/8} \binom{|S|}{k} \le e^{-\gamma k / 8} \binom{(1 + 2^{11} \varepsilon)\lambda k / 2}{k} \stackrel{\eqref{eq:a+c_b}}{\leq} e^{-\gamma k / 8 + 2^{12} \varepsilon k} \binom{\lambda k / 2}{k} < e^{-\gamma k / 10} \binom{\lambda k / 2}{k},
        $$
        where the last inequality follows from setting $C_\gamma$ large enough compared to $L$. The same bound holds for $\mB^-_S$, and we conclude
        $$
            |\mB| \le \sum_{S \in \mC_p} |\mB^+_S| + |\mB^-_S| \le 2 |\mC| \cdot e^{-\gamma k / 10} \binom{\lambda k / 2}{k} \stackrel{\eqref{eq:eps_choice}}{=} o_{n \to \infty}\left( \binom{\lambda k / 2}{k} \right).
        $$

    \end{proof}

    For an arithmetic progression $P$, let $\widehat{P}$ denote the arithmetic progression obtained from $P$ by extending it by $\gamma \lambda k$ elements in each direction. We now show
    \begin{equation} \label{eq:final_bound}
        \left| \left\{ A \in \mF_p \setminus \mB \; : \; A \not \subseteq \widehat{P_A} \right\} \right| = o_{n \to \infty}\left( \binom{\lambda k / 2}{k} \right),
    \end{equation}
    which finishes the proof of the theorem. The following claim is at the heart of our analysis.
    \begin{claim} \label{claim:heart}
        Let $A \in \mF_p \setminus \mB$ such that $A \nsubseteq \widehat{P_A}$, and let $A' = A \cap P_A$. Then 
        $$
            |A' + A'| \le \lambda k -1  - \gamma k / 8.
        $$
    \end{claim} 
    \begin{proof}
        Since $A \not \subseteq \widehat{P_A}$, then at least one of the following is true:
        \begin{enumerate}[(1)]
            \item  $\max(A) > \max(\widehat{P_A})$,
            \item  $\min(A) < \min(\widehat{P_A})$, or
            \item there exists $b \in A \setminus \widehat{P_A}$ such that $\min(\widehat{P_A}) \le b \le \max(\widehat{P_A})$.
        \end{enumerate}
        
        Suppose the property (1) holds. Let $A^+ = A \cap P_A^+$, and write $b := \max(A)$. Then all the elements in $P_A^+ + \{b\}$ are larger than $2 \max(P_A)$, thus
        $$
            (A' + A') \cap (P_A^+ + \{b\}) = \emptyset.
        $$
        Then
        $$
            |A' + A'| \le |A + A| - |A^+| \le \lambda k - |A^+| \le \lambda k - 1 - \gamma k / 8,
        $$
        where the last inequality follows from the assumption that $A \not \in \mB$.

        The same argument applies if the second property holds. Finally, suppose that the third property holds but not the first two, and let $b \in A \setminus \widehat{P_A}$ be such that $\min(\widehat{P_A}) \le b \le \max(\widehat{P_A})$. Let $d \in \mathbb{N}$ denote the step of the progression $P_A$ and $a = \min(\widehat{P_A})$. Then $A' + A' \subseteq 2a + d \mathbb{Z}$ and $A' + b \subseteq a + b + d \mathbb{Z}$. Since $b \not \equiv a \pmod{d}$, we conclude $(A' + A') \cap (A' + \{b\}) = \emptyset$ and the claim again follows.
    \end{proof}

    For an arithmetic progression $P$, let
    $$ 
        \mA(P,j) := \{ A \subseteq P: |A| = k - j \text{ and } |(P+P) \setminus (A+A)| \geq \gamma k / 8 \},
    $$
    and note that $|\mA(P_S, j)| = |\mA([\lambda k / 2], j)|$. Furthermore, $|P_S + P_S| = \lambda k-1$. Claim \ref{claim:heart} thus justifies the following upper bound: 
    \begin{align}
    \label{eq: A bound}
         \left| \left\{ A \in \mF_p \setminus \mB \; : \; A \not \subseteq \widehat{P_A} \right\} \right| &\le \sum_{S \in \mC_p} \sum_{j = 1}^{|S \setminus P_S|}  \binom{|S \setminus P_S|}{j} \left| \mA(P_S, j)\right| \nonumber\\
        &= \sum_{S \in \mC_p} \sum_{j = 1}^{|S \setminus P_S|} \binom{|S \setminus P_S|}{j}  \left| \mA([\lambda k / 2], j) \right|.
    \end{align}
     To estimate $\left| \mA([\lambda k / 2], j) \right|$ for $j \in [2^{10} \varepsilon \lambda k]$, we apply Corollary \ref{cor: k-set probability} with parameters $n_c = \lambda k/2$, $m_c = \gamma k/8 $, $k_c = k-j$ and $\varepsilon_c = 1/4$. We are able to use the corollary as 
    \[ \frac{C'}{\sqrt{m_c}} \leq \frac{k_c}{n_c} \leq \frac{1}{2}, \]
    where $C'$ is a large constant involving the absolute constant given by the corollary.
    Indeed, the upper bound follows from 
\[ \frac{k-j}{\lambda k/2} \leq \frac{2}{\lambda}  \leq \frac{1}{2},  \]
and the lower bound is satisfied whenever 
    \[ 2 C'\lambda\sqrt{\frac{k}{\gamma}} \leq k-j. \]
    This is true because $j \leq 2^{10} \varepsilon \lambda k = o(k)$, and
    \[ 2C'\lambda\sqrt{\frac{ k}{\gamma}} \leq 2 \lambda^{1/3} k^{2/3} \log^{-1/3} n = o(k). \]
    Thus we find for $A' \subseteq [\lambda k/2]$ being a $(k-j)$-element subset, drawn uniformly at random among all such sets,
    \[ \Pr( |[\lambda k] \setminus (A' + A')| > \gamma k/8) \leq  \left( 1 - \frac{k-j}{\lambda k/2} \right)^{\gamma k/32} \leq \exp \left( - \frac{ \gamma k}{16 \lambda} + \frac{ j \gamma }{16 \lambda} \right) \leq e^{-\gamma k/(32 \lambda)}. \]
    Therefore,
    $$
        \left| \mA(P_S, j) \right|  = \left| \mA([\lambda k / 2], j) \right| \leq e^{-\gamma k/(32 \lambda)} \binom{\lambda k/2}{k-j} = e^{-\gamma k/(32 \lambda)} \binom{|P_S|}{k - j}.
    $$
    Using $|S \setminus P_S| \leq 2^{10} \varepsilon \lambda k$, the bound in \eqref{eq: A bound} becomes
    \begin{align*}
         \left| \left\{ A \in \mF_p \setminus \mB \; : \; A \not \subseteq \widehat{P_A} \right\} \right| &\le \sum_{S \in \mC_p} \sum_{j = 1}^{|S \setminus P_S|} e^{- \gamma k / (32 \lambda)} \binom{|S \setminus P_S|}{j} \binom{|P_S|}{k - j} \\
         &\stackrel{\eqref{eq:a_c-d}}{\le} e^{-\gamma k / (32 \lambda)}\; \sum_{S \in \mC_p} \sum_{j = 1}^{2^{10} \varepsilon \lambda k} \left( \frac{2^{13} \varepsilon k}{j} \right)^j \binom{|P_S|}{k}.
    \end{align*}
    As 
    $$
        \sum_{j = 1}^{2^{10} \varepsilon \lambda k} \left( \frac{2^{13} \varepsilon k}{j} \right)^j < e^{C' \varepsilon k}
    $$
    for some absolute constant $C'$, we conclude
    $$
        \left| \left\{ A \in \mF_p \setminus \mB \; : \; A \not \subseteq \widehat{P_A} \right\} \right| \le e^{-\gamma k / (32 \lambda) + C' \varepsilon k} \; |\mC| \; \binom{\lambda k / 2}{k}.
    $$
    Applying the upper bound on $|\mC|$ and the fact that $C_\gamma$ is sufficiently large yields \eqref{eq:final_bound}.
\end{proof}

\paragraph{AI statement.} AI models were used solely for proofreading purposes. All proofs and the writing are the author's own.

\bibliographystyle{abbrv}
\bibliography{references}

@article{campos2020number,
  title={On the number of sets with a given doubling constant},
  author={Campos, Marcelo},
  journal={Israel Journal of Mathematics},
  volume={236},
  number={2},
  pages={711--726},
  year={2020},
  publisher={Springer}
}

@article{liu2024number,
  title={On the number of sets with small sumset},
  author={Liu, Dingyuan and Mattos, Let{\'\i}cia and Szab{\'o}, Tibor},
  journal={Israel Journal of Mathematics},
  note={To appear.}
}

@article{pollard1974generalisation,
  title={A generalisation of the theorem of {C}auchy and {D}avenport},
  author={Pollard, John M},
  journal={Journal of the London Mathematical Society},
  volume={2},
  number={3},
  pages={460--462},
  year={1974},
  publisher={Oxford University Press}
}

@article{campos2022probabilistic,
  title={The typical structure of sets with small sumset},
  author={Campos, Marcelo and Collares, Maur{\'\i}cio and Morris, Robert and Morrison, Natasha and Souza, Victor},
  journal={International Mathematics Research Notices},
  volume={2022},
  number={14},
  pages={11011--11055},
  year={2022},
  publisher={Oxford University Press}
}

@article{freiman1999structure,
  title={Structure theory of set addition},
  author={Freiman, Gregory A},
  journal={Astérisque},
  volume={258},
  pages={1--20},
  year={1999},
  publisher={Centre National de la Recherche Scientifique}
}

@article{ruzsa1994generalized,
  title={Generalized arithmetical progressions and sumsets},
  author={Ruzsa, Imre Z},
  journal={Acta Mathematica Hungarica},
  volume={65},
  number={4},
  pages={379--388},
  year={1994},
  publisher={Springer}
}

@article{green2004cameron,
  title={The {C}ameron--{E}rd{\H{o}}s conjecture},
  author={Green, Ben},
  journal={Bulletin of the London Mathematical Society},
  volume={36},
  number={6},
  pages={769--778},
  year={2004},
  publisher={Oxford University Press}
}

@article{Sapzhenko2003,
author = {Sapozhenko, Alexander},
year = {2003},
pages = {438-441},
title = {The {C}ameron--{E}rd{\H{o}}s Conjecture},
volume = {68},
journal = {Doklady Mathematics}
}

@article{alon2014refinement,
  title={A refinement of the {C}ameron--{E}rd{\H{o}}s conjecture},
  author={Alon, Noga and Balogh, J{\'o}zsef and Morris, Robert and Samotij, Wojciech},
  journal={Proceedings of the London Mathematical Society},
  volume={108},
  number={1},
  pages={44--72},
  year={2014},
  publisher={Oxford University Press}
}

@article{green2016counting,
  title={Counting sets with small sumset and applications},
  author={Green, Ben and Morris, Robert},
  journal={Combinatorica},
  volume={36},
  pages={129--159},
  year={2016},
  publisher={Springer}
}

@article{chang2002polynomial,
  title={A polynomial bound in {F}reiman's theorem},
  author={Chang, Mei-Chu},
  journal={Duke Mathematical Journal},
  volume={115},
  number={1},
  pages={399--419},
  year={2002}
}

@article{schoen2011near,
 author = {Schoen, Tomasz},
 title = {Near optimal bounds in {Freiman}'s theorem},
 fjournal = {Duke Mathematical Journal},
 journal = {Duke Math. J.},
 issn = {0012-7094},
 volume = {158},
 number = {1},
 pages = {1--12},
 year = {2011},
 language = {English},
 doi = {10.1215/00127094-1276283},
 url = {semanticscholar.org/paper/9d84d3f9a1289eb6803754336016d341c034efb3},
 zbMATH = {5904309},
 Zbl = {1242.11074}
}

@article{sanders2012bogolyubov,
  title={On the {B}ogolyubov--{R}uzsa lemma},
  author={Sanders, Tom},
  journal={Analysis \& PDE},
  volume={5},
  number={3},
  pages={627--655},
  year={2012},
  publisher={Mathematical Sciences Publishers}
}

@article{sanders2013structure,
  title={The structure theory of set addition revisited},
  author={Sanders, Tom},
  journal={Bulletin of the American Mathematical Society},
  volume={50},
  number={1},
  pages={93--127},
  year={2013}
}

@article{green2005counting,
  title={Counting sets with small sumset, and the clique number of random {C}ayley graphs},
  author={Green, Ben},
  journal={Combinatorica},
  volume={25},
  pages={307--326},
  year={2005},
  publisher={Springer}
}

@article{cameron1990number,
 author = {Cameron, Peter J. and Erd{\H{o}}s, Paul},
  title={On the number of sets of integers with various properties},
  journal={Number Theory (R. A. Mollin, ed.)},
  pages={61--79},
  year={1990}
}

@article{alon2014counting,
 author = {Alon, Noga and Balogh, J{\'o}zsef and Morris, Robert and Samotij, Wojciech},
 title = {Counting sum-free sets in abelian groups},
 journal = {Israel Journal of Mathematics},
 volume = {199},
 pages = {309--344},
 year = {2014},
}

@article{alon2025random,
  title={Random {C}ayley graphs and random sumsets},
  author={Alon, Noga and Pham, Huy Tuan},
  journal={arXiv preprint arXiv:2509.02561},
  year={2025}
}

\end{document}